\documentclass[12pt,reqno]{amsart}
\usepackage{xr}
\usepackage{amsthm}
\usepackage[mathscr]{eucal}
\usepackage{amsfonts}
\usepackage{amsmath}
\usepackage{amssymb}
\usepackage{graphicx}
\usepackage{xcolor}

\usepackage{pstricks}
\usepackage{fullpage}
\usepackage{hyperref}
\usepackage{booktabs}
\usepackage{colortbl}

\usepackage{enumitem}

\usepackage{caption}
\usepackage{subcaption}
\usepackage{setspace}
\usepackage{tabularx}
\usepackage{natbib}
\usepackage{booktabs}

\newcommand{\eps}{\epsilon}
\newcommand{\veps}{\varepsilon}

\newcommand{\mb}[1]{\mathbb{#1}}
\newcommand{\mc}[1]{\mathcal{#1}}

\newcommand{\bs}[1]{\boldsymbol{#1}}

\newcommand{\msf}[1]{\mathsf{#1}}

\newcommand{\tn}[1]{\textnormal{#1}}
\newcommand{\inner}[2]{\langle #1,#2\rangle}

\newcommand{\ind}{{1\hspace{-2.5pt}\tn{l}}}

\newcommand{\R}{\mb{R}}
\newcommand{\Z}{\mb{Z}}

\newcommand{\wh}[1]{\widehat{#1}}

\newcommand{\ME}{\msf{E}}

\newcommand{\VAR}{\msf{VAR}}

\newcommand{\COV}{\msf{COV}}

\newcommand{\BIAS}{\msf{BIAS}}

\newcommand{\lf}[1]{\large{#1}}
\newcommand{\LF}[1]{\Large{#1}}

\newenvironment{pf}[1][Proof]{\begin{proof}[\textit{\textbf{#1}}]} {\end{proof}} 
\theoremstyle{plain}
\newtheorem{theorem}{Theorem}
\newtheorem{lemma}{Lemma} 
\newtheorem{proposition}{Proposition} 
\newtheorem{corollary}{Corollary}
\theoremstyle{definition}
\newtheorem{exam}{Example}
\newtheorem{remark}{Remark}
\title{Autocovariance function estimation via difference schemes for a semiparametric 
change point model with $m$-dependent errors}

\author{Michael Levine$^{(1)}$}
\email{mlevins@purdue.edu}

\address{$^1$Department of Statistics\\ 
Purdue University\\ 
West Lafayette, IN 47907}

\author{Inder Tecuapetla-G\'omez$^{(2,3)}$}
\email[Inder~Tecuapetla]{itecuapetla@conabio.gob.mx}
\address{$^2$Direcci\'on del Programa Investigadoras e Investigadores por M\'exico del CONACyT\\
Consejo Nacional de Ciencia y Tecnolog\'ia (CONACyT)\\
Av.~Insurgentes Sur 1582\\
Col.~Cr\'edito Constructor, Benito Ju\'arez 03940, Ciudad de M\'exico}

\address{$^3$Direcci\'on de Geom\'atica\\
Comisi\'on Nacional para el Conocimiento y Uso de la Biodiversidad (CONABIO)\\
Liga Perif\'erico-Insurgentes Sur 4903\\
Parques del Pedregal, Tlalpan 14010, Ciudad de M\'exico}

\begin{document}
%
\begin{abstract}
We discuss a broad class of difference-based estimators of the autocovariance function
in a semiparametric regression model where the signal consists of the sum
of an identifiable smooth function and another function with jumps (change points)
while the errors are $m$-dependent. We establish that the influence of the smooth part of the signal
over the bias of our estimators is negligible; this is a general result as it does not
depend on the distribution of the errors.
We show that the influence of the unknown smooth function is negligible also in the
mean squared error (MSE) of our estimators. Although we assume Gaussian errors to 
derive the latter result, our finite sample studies suggest that the class of proposed
estimators still show small MSE when the errors are not Gaussian.
Our simulation study also demonstrates that, when the error process is misspecified as an AR$(1)$ 
instead of an $m$-dependent process, our proposed method can estimate autocovariances about 
as well as some methods specifically designed for the AR($1$) case, and sometimes even better 
than them.

We also allow both the number of change points
and the magnitude of the largest jump grow with the sample size. In this case, 
we provide conditions on the interplay between the growth rate of these two quantities,
and the vanishing rate at zero of the modulus of continuity of the smooth part of the regression function, that ensure $\sqrt{n}$ consistency of our autocovariance estimators.

As an application, we use our approach to provide a better understanding of the
possible autocovariance structure of a time series of global averaged annual temperature anomalies.
Finally, the \texttt{R} package \texttt{dbacf} complements this paper. 

\end{abstract}
%
%
\maketitle
%



\section{Introduction}~\label{sec.introMain}
\setcounter{section}{1} 
\setcounter{equation}{0} 


Let us begin by considering the nonparametric regression model with correlated errors
\begin{equation}\label{model1}
    y_i = g(x_i) + \varepsilon,\quad i=1,\ldots,n,
\end{equation}
where $x_i$ are the fixed sampling points, $g$ is an unknown mean function
that can be discontinuous, e.g. a change point model or a signal with monotonic trend, and $(\varepsilon)$ is a zero mean 
stationary time series error process. For such a model, the knowledge of the autocovariance function (ACF)
$\gamma_{h}=\ME[\eps_{0}\eps_{h}]$, $h=0,1,\ldots$ is essential. For instance, accounting for
an appropriate estimate of the long-run variance ($\sigma_\ast^2 = \sum_{k\in \Z}\gamma_k$)
plays a crucial role for developing multiscale statistics aiming to either estimate the total number
of change points, cf.~\cite{Dette.etal.2018}, or test for local changes in an apparent nonparametric 
trend, cf.~\cite{Khismatullina.Vogt.20}. Similar models have been considered in \cite{Davis.etal.06},
\cite{Jandhyala.13}, \cite{Preuss.etal.14}, \cite{vogelsang2016exactly} and \cite{Chakar.etal.16} among many others.
Generally speaking, ACF estimates are important for bandwidth selection, confidence interval
construction and other inferencial procedures associated with nonparametric modeling, cf.~\cite{Opsomer.etal.01}.
Some of these authors have considered parametric error structures such as ARMA($p$, $q$) or ARCH(1,1)
models. 
In this manuscript, instead of considering a specific error process model, such as e.g. ARMA($p$,$q$), we will consider an $m$-dependent error structure and and a mean function that consists of both smooth and discontinuous parts. This is also the mean structure considered in \cite{chan2022optimal}.

More specifically, we consider the general regression model with correlated errors
\begin{equation}~\label{eq.partialModel}
    y_i = f(x_i) + g(x_i) + \varepsilon_i, \qquad i = 1,\ldots,n,
\end{equation}
where we assume that $f$ is an unknown continuous function on $[0,1]$,
$x_i = i/n$ are sampling points, $g$ is the stepwise constant function
\begin{equation}\label{piecewise}
    g(x) = \sum_{j=1}^{K}\,a_{j-1} \ind_{[\tau_{j-1}, \tau_{j})}(x), \quad x\in [0,1),
\end{equation}
with $a_{j}\ne a_{j+1}$ and change points at $0=\tau_{0}<\tau_{1} < \cdots < \tau_{K-1}<\tau_{K}=1$;
the levels $(a_j)$, the number of change points $(K)$ and their location $(\tau_j)$
are all unknown. 
We will assume that 
the errors $(\varepsilon_i)$ form a zero mean, stationary, $m$-dependent process, i.e., 
we assume that $\gamma_h = 0$ when $|h| > m$.
To ensure identifiability, we require that $\int_{0}^{1}f(x)\,dx=0$. 

Somewhat similar partial linear regression models (where the function $g(x)$ is a linear function of $x$ that 
does not contain jumps) with correlated errors have a fairly long history in statistical research. 
\cite{engle1986semiparametric} already established, in their study of the effect of weather on electricity 
demand that the data were autocorrelated at order one. \cite{gao1995asymptotic} was probably the first to 
study estimation of the partial linear model with correlated errors. \cite{you2007semiparametric} obtained 
an improved estimator of the linear component in such a model using the estimated autocorrelation structure 
of the process $(\varepsilon)$. 

A model where the mean structure is exactly the same as in \eqref{eq.partialModel} but the errors are iid is 
typically called a \emph{Nonparametric Jump Regression} (NJRM) and was considered in \cite{qiu1998local} 
who were concerned with the jump detection in that model. This model is often appropriate when the mean function 
in a regression model jumps up or down under the influence of some important random events. 
Good practical examples are stock market indices, physiological responses to stimuli and many others. 
Regression functions with jumps are typically more appropriate than continuous regression models for such data. 
A method that can be used to estimate the location curves and surfaces for $2$- and $3$- dimensional versions 
of NJRM was proposed in \cite{chu2012estimation}. Further generalizations to the case where the observed image 
also experiences some spatial blur but the pointwise error remains serially uncorrelated are also available, 
see e.g.~\cite{kang2014jump} and \cite{qiu2015blind}. With this background in mind, our model \eqref{eq.partialModel} 
effectively amounts to the generalization of the 1-dimensional NJRM to the case of serially correlated errors. 

$m$-dependency may be construed as a restrictive model
for correlated structures, 
however, $m$-dependency
is an appropriate proxy for more elaborated dependency measures, within the framework of model~\eqref{model1}, 
provided that the corresponding autocovariance function decays exponentially fast, 
see Section~4 of \cite{Dette.etal.2018}. Since the appearence of the concept of
physical dependence measure, cf.~\cite{Wu2005}, there has been an increasing theoretical interest
for using $m$-dependency to approximate general dependence structures, \cite{Berkes.etal.2009},
\cite{Berkes.etal.2013}, \cite{dette2023}.
 {Note also that there is a number of practically important cases where the value of $m$ may be known beforehand as in e.g. \cite{Hotz.etal.13} and \cite{pein2018fully}.}
Thus the relevance of this work lies precisely in providing a family of ACF estimators that 
circumvent the difficult estimation of a mean function which consists of both a change point 
component and a smooth function $f$. To this end we will focus on the family of difference-based estimators.

Difference-based estimators can be traced back to \emph{the mean successive difference} of \cite{vonNeumann.etal.41}.
Since then this computationally efficient variance estimator has been studied with many different purposes in mind.
For instance, in nonparametric regression with smooth signals and homoscedastic errors, it has been considered for improving 
bandwidth selection (\cite{Rice.84}) and asymptotic efficiency of variance estimators (\cite{Gasser.etal.86} and \cite{Hall.etal.90});
it has also been considered for variance estimation under heteroscedasticity of the errors (\cite{Muller.Stadtmuller.87}).
\cite{Dette.etal.98} have discussed that for a small sample size, a difference-based variance estimator may have a non-negligible 
bias and to overcome this issue, optimal sequences can be employed.

Difference-based estimators have also been considered in smooth nonparametric regression with correlated errors,
e.g.~\cite{Muller.Stadtmuller.88}, \cite{Herrmann.92}, \cite{Hall.VanKeilegom.03} and \cite{Park.etal.06}.
More recently, and close to our work, optimal variance difference-based estimators have been proposed
in the standard partial linear model under homoscedasticity of the errors, e.g.~\cite{Wang.etal.17} and 
\cite{Zhou.etal.18}, among others. \cite{wang2017asymptotically} studied the optimal difference-based estimator for variance. 
\cite{dai2015difference} proposed difference-based variance estimator for repeated measurement data. 
{ \cite{Tecuapetla.Munk.2017} 
studied ACF estimation via difference-based estimators of second order and $(m+1)$-gap 
in the model~\eqref{eq.partialModel} when $f=0$.}
To the best of our knowledge the ACF estimation problem via difference schemes in the 
Eq.~\eqref{eq.partialModel}, has not been considered.

Perhaps the first contribution of this paper is the rather conceptual extension of 
model~\eqref{model1} (considered in \cite{Tecuapetla.Munk.2017}) to model~\eqref{eq.partialModel}.
Note that the setup considered here is not a special case of the setup in Section 4
of \cite{Tecuapetla.Munk.2017}. There, it is specified that there exists a positive
number $c$ such that all of the jumps are greater than $c$ in absolute value.
Thus, the mean function is guaranteed to have a certain number of discontinuities
$K_n$ that may depend on $n$. On the contrary, the setup consider in this work
does not require the existence of such a constant $c$ and so the existence of a given
number of discontinuities is not guaranteed.
Moreover, in Section 4 of \cite{Tecuapetla.Munk.2017} it is required that, on every
interval $[\tau_j,\tau_{j+1})$, the regression function must be H\"older continuous
with some index $\alpha_j$ whereas in our setup the smooth component $f$ of the regression
function is only continuous.
Additionally, the separation of signal in two parts, 
$g$ which contains jumps and $f$ which is rather smooth, 
allows us to establish all of our main results under the assumption that $f$ is 
continuous on $[0,1]$.


{The contributions of this paper continue with} the explicit derivation of the expected value
of difference-based variance estimator of \emph{arbitrary} order $\ell$ and $(m+1)$-gap, see Eq.~\eqref{eq.dbQF} for
a proper definition of this class of estimators and Theorem~\ref{theo.DBE-partial} for the result itself. 
An immediate consequence of this result is that, if in the semiparametric change point model given by \eqref{eq.partialModel}
we assume that the largest jump is bounded (the bound does not depend on $n$) and the total 
variation of the function $g$ is of order $o(n)$, 
then any estimator within the class~\eqref{eq.dbQF} is asymptotically unbiased. 
We stress that Theorem~\ref{theo.DBE-partial} does not require us to assume a specific distributional family for the error model.
As explained in the remarks following Theorem~\ref{theo.DBE-partial},
an increment in the (arbitrary) order $\ell$ increases the bias of any member of
the variance estimator class~\eqref{eq.dbQF};
the magnitude of this increase is of the order of the product of $m+1$ times
the quadratic variation of the function $g$ times an increasing function depending on $\ell$. 
Observe that among the latter quantities we can only control the lag order $\ell$. 
Consequently, we opt for studying some asymptotic properties
of the member of class~\eqref{eq.dbQF} with the smallest (finite sample) bias, namely, the variance estimator of order 1,
\begin{equation}\label{eq.gamma0.hat}
	\wh{\gamma}_0 = \frac{1}{2\,n_{m+1}}\sum_{i=1}^{n_{m+1}}\left( y_i - y_{i + (m+1)} \right)^2,
	\quad
	n_{m+1} = n - (m+1).
\end{equation}
In order to estimate the remaining values of the autocovariance function 
we focus on the estimator
\begin{equation}\label{eq.gamma.h.hat}
	\wh{\gamma}_h 
	= 
	\wh{\gamma}_0 - \frac{1}{2(n-h)}\sum_{i=1}^{n-h}\left( y_i - y_{i + h} \right)^2,\quad h = 1,\ldots, m.
\end{equation}
Also from the remarks of Theorem~\ref{theo.DBE-partial} it follows that 
in the semiparametric change point model with $m$-dependent
stationary errors given by \eqref{eq.partialModel} 
even when the number of change points grows at rate $o(n^{\epsilon})$ and the 
size of the largest jump grows at rate $o(n^{-\epsilon})$ then, for any $\epsilon > 0$, our autocovariance
estimators \eqref{eq.gamma0.hat}-\eqref{eq.gamma.h.hat} are $\sqrt{n}$ consistent provided
that the modulus of continuity vanishes at zero at the rate $o(n^{-1/4})$ as $n\to \infty$, see Theorem~\ref{eq.root-n}. 
Among our main results, this is the only one in which we utilized a distributional form for the error model (Gaussian).
Theorem~\ref{eq.root-n} is the second major contribution of this paper.

Section~\ref{sec.resultsGeneral} contains preliminary
calculations needed to establish the main results of this paper. In Section~\ref{sec.sim} 
we conducted several numerical studies to assess the accuracy and precision of our estimators in the setting of 
model \eqref{eq.partialModel}-\eqref{piecewise}; we also assess the robustness of our method by considering
$m$-dependent errors with heavy-tailed marginal distribution. 
Separately, we consider the robustness of our method to reasonably small violations of the $m$-dependency requirement for the error process. This is done by using our method to estimate autocovariances when the true error process is an AR($1$) process. This set-up is meaningful since the autocovariances of the AR($1$) process at lag $k$ are proportional to 
$\phi^{|k|}$.
Thus, if $\phi$ is chosen to be small in absolute value, the autocovariances of AR($1$) process decay quite quickly and become very small after a finite number of lags. We compare the performance of our method under the AR($1$) error assumption to two alternative methods. First, introduced in the paper \cite{Hall.VanKeilegom.03}, was designed specifically for the case of a smooth mean function. The second method is that of \cite{Chakar.etal.16}. That method was designed as a robust estimator of the autoregressive parameter of an AR(1) error process in a model with a mean that has a finite number of breakpoints. Since this method is not meant to provide a direct estimate of the autocovariance function at lag $k$, this function has to be obtained using standard formulas that connect the autoregressive coefficient of an AR($1$) process and its autocovariance at lag $k$ given in any introductory time series text e.g.~\cite{shumway2019time} p.~$77$. All of our simulations are based on functions provided by the R package 
\href{https://cran.r-project.org/package=dbacf}{dbacf}, available on The Comprehensive R Archive 
Network, of \cite{dbacf}.

{Section~\ref{sec.apps} shows that our method can be employed to provide a better understanding of the ACF structure of global, averaged, annual temperature anomalies spanning
from 1808 to 2021.}  This dataset can be found
in the package \texttt{astsa} by \cite{Stoffer.Poison.23} developed under the \texttt{R} language
for statistical computing, cf.~\cite{R}.
Finally, some technical details used in our proofs are relegated to Appendices~\ref{sec.supp}
and \ref{sec.appendix}.

%


\section{Main results}~\label{sec.mainResults}

Before introducing the class of difference-based estimators we require some notation.
First, for any $i<j,$ we will use the notation $i:j$ for an index 
vector $(i,i+1,\ldots,j)$. Thus, for a generic vector $\bs{v}^\top = ( v_1,\, v_2, \, \ldots, \, v_n )$, 
for $1\leq i < j \leq n$, we define $\bs{v}_{i:j}^\top = ( v_i, \, v_{i+1}, \, \ldots, \, v_j )$. 
Also, $f_i$ and $g_i$ will denote $f(x_i)$ and $g(x_i)$.
Thus, $\bs{f}_{i:j}^\top$, $\bs{g}_{i:j}^\top$ and $\bs{\varepsilon}_{i:j}^\top$ denote the vectors 
$( f(x_i), \, \ldots, \, f(x_j) )$,
$( g(x_i), \, \ldots, \, g(x_j) )$ and 
$( \varepsilon_i, \, \ldots, \, \varepsilon_j )$,
respectively.  
The quadratic variation of the function $g(\cdot)$ will be denoted by $J_{K}:=\sum_{j=0}^{K-1}(a_{j+1}-a_{j})^{2}$. 
For vectors $\bs{v}$ and $\bs{w}$, $\inner{\bs{v}}{\bs{w}}$ denotes their Euclidean inner product.
From now on $f:[0,1]\to \R$ is a continuous function. Such a function is, of course, a uniformly
continuous one. Due to this, there exists a function
$\omega(\cdot)$, called a modulus of uniform continuity, such that $|f(x)-f(y)|\leq \omega(|x-y|)$ for all $x,y\in [0,1]$. This function is vanishing at zero, right-continuous at zero, and strictly increasing
over the positive half of the real line. The function 
$g$ obeys \eqref{piecewise}. 

Let $\ell\geq 1$ be given such that $n_{\ell_{m+1}} = n - \ell(m+1) \gg 0$. It is known that in the change point
regression with $m$-dependent errors model, which is a particular case of \eqref{eq.partialModel},
in order to get a consistent variance estimator based on difference schemes it is necessary to
consider observations which are separated in time by at least $m+1$ units, cf.~Theorem~5 of
\cite{Tecuapetla.Munk.2017}. That is, a consistent variance difference-based estimator must consider 
\emph{gaps} of size (at least) $m+1$ observations. Due to this we utilize the vector of weights 
$\bs{d}^\top = ( d_0, \, d_1, \, \ldots, \, d_\ell )$ and define a difference of 
order $\ell$ and a $(m+1)$-gap as $\Delta_{\ell,m+1}(y_i; \bs{d}) = \sum_{s=0}^\ell\,d_s\,y_{i + s(m+1)}$. 
We utilize this object to define a difference-based variance estimator of order $\ell$ 
and $(m+1)$-gap as the quadratic form
\begin{equation}~\label{eq.dbQF}
    \bs{Q}_{\ell,m+1}(\bs{y};\bs{d}) 
    =
    \frac{1}{p(\bs{d})\,n_{\ell_{m+1}}}\sum_{i=1}^{n_{\ell_{m+1}}}\Delta_{\ell,m+1}^2(y_i; \bs{d}), 
    \quad
    p(\bs{d}) = \sum_{s=0}^{\ell}\,d_s^2.
\end{equation}
Throughout the paper, in order to simplify notation, we may omit $\bs{d}$ inside the parentheses for 
observation differences (and quadratic forms) and simply write $\Delta_{\ell,m+1}(y_i)$ 
(and $\bs{Q}_{\ell,m+1}(\bs{y})$), unless any confusion results from such an omission.

\begin{exam}
For $m>0$, $\ell=1$, $d_0=1$ and $d_1=-1$, $\Delta_{1,m+1}^2(\varepsilon_i;\bs{d})$ becomes 
$\varepsilon_{i}^2-2\varepsilon_i\,\varepsilon_{i+m+1}+\varepsilon_{i+m+1}^2$. 
$m$-dependence guarantees that $\ME [\varepsilon_i\,\varepsilon_{i+m+1}]=0$ 
for any $i$. Hence, the expected value of the \emph{core statistic} $\Delta_{1,m+1}^2(\varepsilon_i;d)$ 
is equal to $2\sigma^2$. For the general difference-based estimator given by Eq.~\eqref{eq.dbQF}, 
the expected value of $\Delta_{\ell,m+1}^2(\varepsilon_i; \bs{d})$ is equal to $p(\bs{d})\,\sigma^2$.
\end{exam}

In order to obtain a clear representation of the part of the bias of $\bs{Q}_{\ell,m+1}(\bs{y})$
which is directly linked to the stepwise constant function $g$ we assume the following 
restriction on the distance between the jumps,
\begin{equation}\label{eq.DistBetweenJumps}
    \min_{1\leq i\leq {K-1}}| \tau_{i+1} - \tau_{i} | > \ell(m+1)/n.
\end{equation}
Observe that this condition ties up together the distance between change points,
the depth of dependence $m$ and the lag order $\ell$. 
Whether there exist alternative conditions
to \eqref{eq.DistBetweenJumps} which are less restrictive and
allow us to get a simple description of this bias term is an interesting research question
which is outside the scope of the present work.\medskip 

\begin{theorem}~\label{theo.DBE-partial}
 Consider the semiparametric change point regression model with zero mean, stationary, $m$-dependent errors
 defined through Eqs.~\eqref{eq.partialModel}-\eqref{piecewise} and
 assume that the condition~\eqref{eq.DistBetweenJumps} is satisfied. Define
 \[
	J_K^{||} = \sum_{j=0}^{K-1}\,| a_j - a_{j+1} |.
 \]
    Assume further that 
    \begin{equation}~\label{eq.DBE-partial.conditions}
     \sum_{s=0}^\ell\,d_s = 0 \quad \mbox{ and } \quad p(\bs{d}) = 1. 
    \end{equation}    
    Then,
    \begin{align*}
        \ME[ \bs{Q}_{\ell,m+1}(\bs{y}) ]
        &=
        \gamma_0 
        + 
        o\left( \frac{1}{n_{\ell_{m+1}}}\, \omega^2\left( \frac{m+1}{n} \right) \right)
        +
        o\left( \omega^2\left( \frac{m+1}{n} \right) \right)\\
        &+
        o\left( \omega\left( \frac{m+1}{n} \right)\,\frac{(m+1)R_\ell(\bs{d})\,J_K^{||}}{n_{\ell_{m+1}}} \right)
        +
        o\left( \frac{(m+1)P_\ell(\bs{d})\,J_K}{n_{\ell_{m+1}}}  \right),
    \end{align*}
	where
	\begin{align*}
	R_\ell(\bs{d})
	&=
	\sum_{r=0}^{\ell-1}\,\sum_{s=0}^{r}|d_s|\,
	\sum_{r=0}^{\ell-1}(\ell-r)\,\sum_{s=0}^{r}|d_s|,\\
	P_\ell(\bs{d})
	&=
    \sum_{r=0}^{\ell-1}(\ell-r)\left( \sum_{s=0}^r d_s \right)^2
    +
    2\times\ind_{[2,\infty)}(\ell)\times
    \left(
    \sum_{r=0}^{\ell-2}(\ell-1-r)
    \sum_{s=0}^r\,d_s\,\sum_{p=r+1}^{\ell-1}\sum_{q=0}^{s}\,d_q
    \right).
	\end{align*}
\end{theorem}\medskip

\begin{remark}~\label{rm1}
Note that the magnitude of the bias depends strongly on the quadratic variation $J_K$. 
Note also that $J_{K}^{||}$ is effectively the total variation of the function 
$g(x)=\sum_{j=1}^{K}a_{j-1} \ind_{[\tau_{j-1}, \tau_{j})}(x)$
and that one can guarantee that 
$J_{K}$ is growing relatively slowly with $n$ by imposing a condition on $J_{K}^{||}$. 
Indeed, it is clear that the quadratic variation 
$J_{K}\le \left(\max_{1\le j \le K}|a_{j-1}-a_{j}|\right)\left(\sum_{j=1}^{K}|a_{j-1}-a_{j}|\right)
=
\max_{1\le j \le K}|a_{j-1}-a_{j}|\, \times\, J_{K}^{||}$.
Therefore, it is enough to impose a bound on the growth rate of $J_{K}^{||}$ and on the growth rate 
of the maximum jump size $\max_{1\le j \le K}|a_{j-1}-a_{j}|$ to guarantee a reasonably low rate of 
growth for $J_{K}$.
\end{remark}
\begin{remark}
{The normalization requirement ~\ref{eq.DBE-partial.conditions} guarantees asymptotic unbiasedness of $\bs{Q}_{\ell,m+1}(\bs{y})$ as an estimator of $\gamma_0$. Moreover, the result obtained in this Theorem is independent of the order of $d$s in the vector $\bs{d}$.}
\end{remark}

Observe that these considerations carry over when in the model~\eqref{eq.partialModel}-\eqref{piecewise}
we allow that the number of change points and maximum jump size depend on $n$. 
More precisely,

\begin{corollary} 
	Suppose that the conditions of Theorem~\ref{theo.DBE-partial} hold. 
	Additionally, suppose that we allow the number of change points to depend on
	the sample size. 
	Suppose further that $\max_{1\le j \le K_n}|a_{j-1}-a_{j}|$
	is bounded by a constant not depending on $n$, and $J_{K_n}^{||} = o(n)$.
	Then $\ME[ \bs{Q}_{\ell, m+1}(\bs{y};\bs{d}) ]\to\gamma_0$.
\end{corollary}

\begin{remark}
Observe that thus far we have not made any distributional assumption on the errors $(\varepsilon_i)$;
null mean, stationarity and $m$-dependence are sufficient to establish that, asymptotically, the 
influence of the smooth part of the regression function, i.e.~$f$, on the bias of the general difference-based 
variance estimator $\bs{Q}_{\ell,m+1}(\bs{y})$ is negligible. Moreover, this result remains true regardless of
the order $\ell$.
\end{remark}

\begin{remark}
Theorem~\ref{theo.DBE-partial} also provides a hint as to what class of autocovariance estimators may be useful in practice. More precisely, note that the quantity $R_{\ell}(d)$ is always non-negative and monotonically increasing as a function of the order $\ell$. The same is true for the first term of $P_{\ell}(d)$ as well. This suggests that, especially for relatively small sample sizes $n$,  it is possible that the increase in the order $\ell$ of the difference-based estimator may increase its bias. Thus, from a practical viewpoint, it may not make a lot of sense to consider 
difference-based estimators of the variance (and autocovariances) of the error process $(\varepsilon_i)$ for 
$\ell > 1$. Moreover, the condition \eqref{eq.DistBetweenJumps} implies that, if we want to use larger $\ell$, it is necessary to impose a more stringent condition on the change points to guarantee the same order of the bias. In other words, we would have to assume that change points are farther apart in such a case which may not always be a realistic assumption.
\end{remark}


In light of the above, the rest of the paper will be devoted to establishing some asymptotic 
properties of the difference-based estimators of first order, $(m+1)$-gap, and weight 
$\bs{d}_1=(1/\sqrt{2}, -1/\sqrt{2})$, $\bs{Q}_{1,m+1}(\bs{y}; \bs{d}_1)$. 
Note that this estimator is equivalent to $\wh{\gamma}_0$ introduced in Eq.~\eqref{eq.gamma0.hat}:
\begin{equation}\label{eq.GammaHat}
\wh\gamma_{0} := \bs{Q}_{1,m+1}(\bs{y}; \bs{d}_1).
\end{equation}
The autocovariances $\gamma_h$ with $h=1,\ldots,m$ will be estimated 
using the following difference of random quadratic forms,
\begin{equation}~\label{eq.GammaHatH}
    \wh{\gamma}_h 
    := 
    \bs{Q}_{1,m+1}(\bs{y};\bs{d}_1) - \bs{Q}_{1,h}(\bs{y}; \bs{d}_1 ),\quad h=1,\ldots,m.  
\end{equation}
Observe that \eqref{eq.gamma.h.hat} and \eqref{eq.GammaHatH} are equivalent.


Remark~\ref{rm1} points at the direction on imposing conditions on the total variation of the
stepwise constant function $g$ in order to obtain appropriate convergence rate of our estimators.
The following result tells us a bit more. Indeed, we can allow the number of change points, $K,$
to depend on the sample size $n$ and yet obtain appropriate rates of convergence for the
estimators \eqref{eq.GammaHat} and \eqref{eq.GammaHatH}. This is possible through an interplay
between the growth rate of the number of change points and the size of the largest jump.
More precisely, \medskip

\begin{theorem}~\label{eq.root-n}
 Let $\epsilon > 0$ be given. Suppose that the conditions of Theorem~\ref{theo.DBE-partial}
 are satisfied with $\ell=1$. 
 Additionally, suppose that we allow the number of change points to depend on the sample size, 
 say we have $K_n$ change points. Assume also that the errors are Gaussian and furthermore
 that
\[
	\omega\left( \frac{m+1}{n} \right) = o(n^{-1/4}),\quad K_n=o(n^\epsilon)\quad \mbox{and }\max_{1\leq j\leq K_n}| a_{j} - a_{j-1} | = o (n^{-\epsilon}).
\]
 Then, for $h=0,\ldots, m$
 \[
  \BIAS ( \sqrt{n} \, \wh{\gamma}_h ) 
  = 
  o(1)\quad \mbox{and }\quad
  \VAR( \sqrt{n} \, \wh{\gamma}_h )
  =
  \mc O (1).
 \]
\end{theorem}
 

In Theorem~\ref{eq.root-n} we have utilized Gaussianity of the errors; this allows us
to compute moments of fourth order of the quadratic forms $\bs{Q}_{1,h}({\bs y})$, 
$h=1,2,\ldots,m$ explicitly.
Moreover, Gaussianity of the errors has been exploited in many influential publications
in the change point literature, see \cite{Fryzlewicz.14} and \cite{Frick.Munk.Sieling.14} among many others.
More recently, Gaussianity has contributed to establishing an explicit form of the asymptotic
minimax detection boundary for a change point model with dependent errors,
cf.~\cite{Enikeeva.etal.2019}.
Rather than being a limitation, these examples argue in favor of Gaussianity as a means to pave the 
way to obtain general results. In our work, assuming normal errors has allowed us to reveal that 
even when the number of change points tends to infinity (at the rate $o(n^{\epsilon})$), and the largest jump grows at the rate $o(n^{-\epsilon})$ 
($\epsilon>0$), there is a $\sqrt{n}$-consistent class of autocovariance function estimators 
in the semiparametric change point model~\eqref{eq.partialModel}-\eqref{piecewise} provided
that the rate at which the modulus of continuity vanishes at zero is $o(n^{-1/4})$.

\section{Asymptotic properties of autocovariance difference-based estimators}~\label{sec.resultsGeneral}

The results of this section are proven in the Appendix~\ref{sec.supp}.
We will denote $n_{m+1}:=n-(m+1)$ and $n_{h}:=n-h,$ $h=1,\ldots,m$.
The following proposition provides expressions for the bias 
of estimators $\wh{\gamma}_h$, $h=0,\ldots, m$; its proof is omitted as 
it can be deduced from that of Theorem~\ref{theo.DBE-partial}.

 \begin{proposition}~\label{prop.1}
    Suppose that the conditions of Theorem~\ref{theo.DBE-partial} are satisfied and assume that $\ell=1$. 
    Then, 
\begin{align*}
	\ME [ \bs{Q}_{1, m+1}( {\bs y}; \bs{d}_1) ]
	&=
	\gamma_0 + o\left( r_{n_{m+1}}( m+1, \omega, J_K, J_K^{||} ) \right)\\
	\ME [ \bs{Q}_{1, h}( {\bs y}; \bs{d}_1) ]
	&=
	\gamma_0-\gamma_h + o\left( r_{n_{h}}( h, \omega, J_K, J_K^{||} ) \right),
\end{align*}
where
\[
	r_n(x, \omega, J_K, J_K^{||})
	=
	\left\{ \left[ \frac{1}{n} + 1 \right]\omega\left( \frac{x}{n} \right) + \frac{x\,J_K^{||}}{\sqrt{2}n} \right\}
	\omega\left( \frac{x}{n} \right)
	+
	\frac{x\,J_K}{2n}.
\]
  
  Consequently, for $h = 0,\ldots,m$,
  \begin{align}~\label{eq3.prop.1}
    \ME[ \wh{\gamma}_h ] 
    &= 
    \gamma_h 
    + 
    o\left( \left\{ \left[\frac{1}{n}+1\right] 
	  \omega\left( \frac{m+1}{n} \right) + \frac{(m+1-h)\,J_K^{||}}{\sqrt{2}\,n} \right\}\,
	  \omega\left( \frac{m+1}{n} \right) \right)\notag\\
	&+
	  o\left( \frac{(m+1-h)\,J_K}{2\,n} \right).
  \end{align}  
\end{proposition}\medskip

\begin{subsection}{On the variance of $\wh{\gamma}_0$}~\label{subSec.variance}

Now we focus on computing the variance of $\wh{\gamma}_0$, or equivalently 
the variance of the difference-based estimator of order 1, gap $m+1$ and
weights $\bs{d}_1=(1/\sqrt{2}, -1/\sqrt{2})$, see Eq.~\eqref{eq.GammaHat}. 
To state the main result of this section,
we first define an additional quantity of our model: 
$H_K^{||}=\sum_{j=1}^{K}\left(t_j-\frac{m+1}{2}\right)|a_{j-1}-a_{j}|$, where
$t_j = \lfloor n \tau_{j} \rfloor$.
\begin{theorem}\label{th_var}
	Suppose that the conditions of Theorem~\ref{theo.DBE-partial} are satisfied and assume that $\ell=1$.
    Additionally, assume that the errors are Gaussian.
Then, 
\begin{align}~\label{eq.VAR.GAMMA0} 
	&\VAR (\wh{\gamma}_0)
	=
	\frac{2m+3}{n_{m+1}}\,\gamma_0^2 \notag\\
	&+
	o\left( 
	\frac{4\omega^2\left(\frac{m+1}{n}\right)}{n_{m+1}} 
	+
	\vartheta_1(n,m)\,\frac{J_K}{n_{m+1}}	
	+
	\vartheta_2(n,m)\,\omega\left(\frac{m+1}{n}\right)\frac{J_K^{||}}{n_{m+1}}
	-
	2\omega\left(\frac{m+1}{n}\right)\mc O \left( \frac{mH_K^{||}}{n_{m+1}^2} \right)
	\right)\,\gamma_0,
\end{align}
where
\begin{align*}
	\vartheta_1(n,m)
	&=
	(m+1)\left( \frac{2m+1}{2n_{m+1}} \right)\\
	\vartheta_2(n,m)
	&=
	(m+1)\left( \frac{(2-\sqrt{2})m + 1}{2n_{m+1}} \right).
\end{align*}
\end{theorem}

\begin{remark}
If $m$ is finite, the main term of the expansion \eqref{eq.VAR.GAMMA0} converges to zero at the rate of $O\left(\frac{1}{n}\right)$. The higher order terms, however, depend, first, on the modulus of continuity of the smooth function $f$ $\omega(\cdot)$ and, second, quantities $J_{K}$, $J_{K}^{||}$, and $H_{K}^{||}$. The latter three quantities reflect the behavior of the function $g$; for example, $J_{K}$ is its quadratic variation, $J_{K}^{||}$ is its absolute variation, and $H_{K}^{||}$ is closely related to the absolute variation. The behavior of the higher order terms is dependent on the interplay between these various quantities. As an example, the third of this terms $\vartheta_2(n,m)\,\omega\left(\frac{m+1}{n}\right)\frac{J_K^{||}}{n_{m+1}}$ suggests that one may have $J_{K}^{||}$ increasing relatively fast with $K$ if the modulus of continuity goes to zero sufficiently quickly as its argument goes to zero. If one is ready, however, to consider $m\rightarrow \infty$ together with $n$, the condition $m=o(n)$ is absolutely necessary to guarantee consistency of the proposed estimator. The higher order terms may necessitate additional requirements, again depending on the relationship between the rate at which $J_{K}$, $J_{K}^{||}$, and $H_{K}^{||}$ go to infinity with $K$, the rate at which $m$ goes to infinity with $n$, and the behavior of the modulus of continuity of the smooth function $f$. 
\end{remark}

\end{subsection}\medskip

\begin{subsection}{On the variance of $\wh{\gamma}_h$}~\label{subSec.covariance}
In this section we characterize the asymptotic behavior of the autocovariance estimator $\wh{\gamma}_h,$ $h=1,\ldots,m$ 
introduced in~\eqref{eq.GammaHatH}. We stress that the proof of the main result
of this section is based on derivations used in the proof of Theorem~\ref{th_var}
presented in Section~\ref{subSec.variance}. Also,
we utilize the series of Lemmas established in Appendix~\ref{sec.appendix}. 
Now, we state a result about the variance of $\wh{\gamma}_h$. 

\begin{theorem}~\label{theo_var_GAMMAh}
Suppose that the conditions of Theorem~\ref{th_var} are satisfied.
Then, 
\begin{align}~\label{eq.VAR.GAMMAh}
    \VAR( \wh{\gamma}_h ) 
    &=
	\frac{\vartheta_3(m,h,\gamma(\cdot))}{n_{m+1}}
	+
	\vartheta_4(n,m,h,\gamma_0)\,\frac{\omega^2(\frac{m+1}{n})}{n_{m+1}}
	+
	\vartheta_5(n,m,h,\gamma(\cdot))\,\frac{J_K}{n_{m+1}}\notag\\
	&+
	\vartheta_6(n,m,h,\gamma(\cdot))\,\omega\left( \frac{m+1}{n} \right)\,\frac{J_K^{||}}{n_{m+1}}
	-
	\omega\left(\frac{m+1}{n}\right)\,\mc O \left( \frac{2mH_K^{||}}{n_{m+1}^2} \right),
\end{align}
where
\begin{align*}
	\vartheta_3(m,h,\gamma(\cdot))
	&=
	(18m+4h-13)\gamma_0^2 + 2(\gamma_0-\gamma_h)^2\\
	\vartheta_4(n,m,h,\gamma_0)\
	&=
	2
	+
	\left( 
	2(m+h) + 4 - \frac{3m+2}{2\sqrt{2}\,n_{m+1}}	
	\right)\\
	\vartheta_5(n,m,h,\gamma(\cdot))
	&=
	2(m+1)
	\left\{ 
	1 + \left[ \frac{a_{m,h}}{16\sqrt{2}(m+1)n_{m+1}} -1\right]\gamma_0
	+
	\left[ 1 - \frac{h}{2(m+1)n_h} \right]\gamma_h	
	\right\}\\
	a_{m,h}
	&=
	(8\sqrt{2}-2)m^2
	+
	[ 12\sqrt{2}-4 + (8\sqrt{2}-2)h ]m\\
	&+
	4\sqrt{2}h^2
	+
	(12\sqrt{2}-2)h
	+
	4\sqrt{2}-2\\
	\vartheta_6(n,m,h,\gamma(\cdot))
	&=
	\frac{b_{m,h}}{2\sqrt{2}(m+1)n_{m+1}}-1\\
	b_{m,h}
	&=
	(2\sqrt{2}-2)m^2
	+
	(2\sqrt{2}-4)m
	+
	2h^2
	+
	(6+2\sqrt{2})mh
	+
	(2\sqrt{2}-1)h
	+
	\sqrt{2}.
\end{align*}
\end{theorem}
\begin{remark}
Note, first, that the main term of the expansion has the rate of $O\left(\frac{1}{n}\right)$ if $m$ is viewed is finite and consists of the first term and part of the third term in the sum \eqref{eq.VAR.GAMMAh}. Taking into account higher order terms,  the rate of convergence of this variance estimator depends, again, on the modulus of continuity of the smooth function $f$ $\omega(\cdot)$ and quantities $J_{K}$, $J_{K}^{||}$, and $H_{K}^{||}$. If $m$ is viewed as infinite it is necessary, yet again, require that $m=o(n)$. Depending on the rates of growth of $J_{K}$, $J_{K}^{||}$, and $H_{K}^{||}$ additional assumptions may have to be imposed on $m$.  
\end{remark}
\end{subsection}

%
\section{Simulations}~\label{sec.sim}

This section contains two simulation studies and in each one of them 
we will employ the autocovariance estimators \eqref{eq.gamma0.hat}-\eqref{eq.gamma.h.hat}
with $d_0=1/\sqrt{2}$, $d_1=-1/\sqrt{2}$.
The first simulation study assesses the performance of these autocovariance estimators for 
a semiparametric change point model with $m$-dependent errors as defined in \eqref{eq.partialModel}-\eqref{piecewise}.
The second study considers the performance of our estimators in the case where the error 
structure is not exactly $m$-dependent but is described by an autoregressive process of order one (AR(1)) with a sufficiently small absolute value of the coefficient. 
This performance is then compared with the performance of the autocovariance structure estimator of \cite{Hall.VanKeilegom.03} and the estimator of \cite{Chakar.etal.16}. The estimator of \cite{Hall.VanKeilegom.03} has been designed 
specifically to handle the case of autoregressive errors in a nonparametric regression model
with a smooth regression signal. The estimator of \cite{Chakar.etal.16} has been designed as a robust estimator of the autoregressive parameter of an AR(1) error process in a model whose mean has a finite number of breakpoints. 
{We utilized the functions of the R package \href{https://cran.r-project.org/web/packages/dbacf/index.html}{dbacf}, available on CRAN, to perform the calculations of this section.}
 
\subsection{Autocovariance structure estimation in a model with $m$-dependent errors}
We perform the first study with two possible stationary distributions of the error process: first, when this stationary distribution is zero mean Gaussian, and second, when that distribution is also zero mean but a non-Gaussian one. The second case is considered in order to assess the robustness of the proposed method against non-normally distributed errors. As an example of a non-Gaussian zero mean error distribution, we choose $t_{4}$, a Student distribution with $4$ degrees of freedom.  

Similarly to \cite{Park.etal.06}, we consider a $1$-dependent error model: $\eps_{i}=r_{0}\delta_{i}+r_{1}\delta_{i-1}$ where $\delta_{i}'$s are i.i.d and distributed either normally or $t_{4}$. It is assumed that $r_{0}=(\sqrt{1+2\nu}+\sqrt{1-2\nu})/2$ and $r_{1}=(\sqrt{1+2\nu}-\sqrt{1-2\nu})/2$ for a parameter $-\frac{1}{2}\le\nu\le \frac{1}{2}.$ Our aim is to estimate the variance and the autocovariance at lags 1 and 2, that is, 
$\gamma_0$, $\gamma_1$ and $\gamma_2$, respectively, 
of the error process $\{\eps_i\}$. The true values of these are clearly $\gamma_0=1,$ $\gamma_1=\nu$ and $\gamma_2=0$.

The simulated signal in this study is a sum of the piecewise constant signal used earlier by 
\cite{Chakar.etal.16} and a smooth function $f(x).$  Briefly, the first additive component is defined as a piecewise constant function $g,$ with six change-points located at fractions $\frac{1}{6}\pm\frac{1}{36},$ $\frac{3}{36}\pm \frac{2}{36},$ and $\frac{5}{6}\pm \frac{3}{36}$ of the sample size $n$. In the first segment, $g=0,$ in the second $g=10,$ and in the remaining segments $g$ alternates between $0$ and $1,$ starting with $g=0$ in the third segment. 
The function $f$ is chosen to ensure that $\int_{0}^{1}f(x)\,dx=0$ in order to satisfy the identifiability constraint. We consider three choices of $f$: a linear function $f_1(x)=1-2x,$ a quadratic function $f_2(x)=4(x-0.5)^{2}-1/3,$ and a periodic function $f_3(x)=\sin(16\pi x).$ All of these functions are defined as zero outside $[0,1]$ interval. The functions are chosen to range from a simple linear function to a periodic function that may potentially increase the influence of the higher order terms in the risk expansions. Tables \ref{tab.mse_plinMethods_ParksErrors1} and \ref{tab.mse_plinMethods_ParksErrors2} summarize the results for $n=1600$ observations obtained from $500$ replications for Gaussian and $t_4$ errors, respectively. The results seem to confirm that, in both cases, MSEs of proposed estimators are rather small and scarcely depend on the choice of the smooth function $f.$ To check consistency of the method, we also performed the same experiment in the case of normal errors using a larger sample size of $n=3000$ observations. The results of this experiment, available in the Table \ref{tab.mse_plinMethods_ParksErrors1i}, show that MSEs decrease for every choice of the smooth function $f$ and for all possible choices of $\nu$ compared to the case of $n=1600.$ This result seems to confirm our conclusion that the main terms in both the squared bias and the variance of proposed estimators do not depend on the choice of $f.$
Finally, it is known that the bias, although negligible asymptotically, may be rather noticeable in smaller sample sizes. This effect has been extensively illustrated in case of homoscedastic nonparametric regression with i.i.d.~data in \cite{Dette.etal.98}. To study the bias magnitude numerically in our case, we also consider two smaller sample sizes $n=500$ and $n=1000$. For brevity, we only consider the situation where the error distribution is a zero mean Gaussian one. The results are illustrated in the Tables~\ref{tab.mse_plinMethods_ParksErrors_1iii} and \ref{tab.mse_plinMethods_ParksErrors1ii}, respectively. First of all, one can see clearly that absolute magnitudes of MSEs in this case are noticeably larger than for larger sample sizes; however, as the sample size increases from $n=500$ to $n=1000$ the magnitude of MSEs goes down noticeably, by a factor of about $4$ in many cases. The dependence on the choice of the function $f$ becomes a little more pronounced for a small sample size of $n=500,$ especially when changing from the choice of $f_2$ to that of $f_3$.

\begin{table}[h]
\caption{The MSE of autocovariance estimators of $\gamma_0=1$, $\gamma_1=\nu$ 
and $\gamma_2=0$ under the $1$-dependent error model where $\delta_{i}$'s are i.i.d. $\mathcal{N}({0},1)$ based on $500$ replications 
of size $1600.$~\label{tab.mse_plinMethods_ParksErrors1}}
\centering
\scalebox{0.45}{
\begin{tabular}{*{22}{c}}

\toprule[1.25pt]
 & \multicolumn{3}{c}{$\nu=-0.5$} & \multicolumn{3}{c}{$\nu=-0.4$} & \multicolumn{3}{c}{$\nu=-0.2$} 
 & \multicolumn{3}{c}{$\nu=0$} & \multicolumn{3}{c}{$\nu=0.2$} & \multicolumn{3}{c}{$\nu=0.4$}
 & \multicolumn{3}{c}{$\nu=0.5$} \\
 \cmidrule[1.25pt]{2-22}
 
 & $\gamma_0$ & $\gamma_1$ & $\gamma_2$ &$\gamma_0$ &  $\gamma_1$ & $\gamma_2$ &$\gamma_0$ &  $\gamma_1$ & $\gamma_2$ & $\gamma_0$ & $\gamma_1$ & $\gamma_2$ & $\gamma_0$
 & $\gamma_1$ & $\gamma_2$ & $\gamma_0$ & $\gamma_1$ & $\gamma_2$ & $\gamma_0$ & $\gamma_1$ & $\gamma_2$ \\
 
 \cmidrule[1.25pt]{2-22}
\lf{$f_1$} & \lf{0.0402} & \lf{0.0383} & \lf{0.0181} & \lf{ 0.0394} & \lf{0.0379} & \lf{0.0180} & \lf{0.0393} 
       & \lf{ 0.0380} & \lf{ 0.0181} & \lf{0.0375} & \lf{ 0.0387} & \lf{0.0185} & \lf{0.0383} & \lf{0.0378}& \lf{0.0181} & \lf{0.0407} & \lf{0.0396} &\lf{0.0182} & \lf{0.0418} & \lf{0.0423} & \lf{0.0200}\\

 \lf{$f_2$} & \lf{0.0390} & \lf{0.0394} & \lf{0.0174} & \lf{0.0397} & \lf{0.0388} & \lf{0.0181} & \lf{0.0405} 
       & \lf{0.0379} & \lf{0.0182} & \lf{0.0404} & \lf{0.0381} & \lf{0.0176} & \lf{0.0406} & \lf{0.0389}& \lf{0.0180}& \lf{0.0405} & \lf{0.0398}& \lf{0.0186} &\lf{0.0408} &\lf{0.0394} &\lf{0.0183}\\

 \lf{$f_3$} & \lf{0.0420} & \lf{0.0424} & \lf{0.0193} & \lf{0.0396} & \lf{0.0441} & \lf{0.0206} & \lf{0.0429} 
            & \lf{0.0421} & \lf{0.0197} & \lf{0.0416} & \lf{0.0422} & \lf{0.0199} & \lf{0.0418} & \lf{0.0432} & \lf{0.0203} & \lf{0.0414} & \lf{ 0.0432} & \lf{0.0208} & \lf{0.0417} &             \lf{0.0427} & \lf{0.0201} \\
 \bottomrule[1.25pt]  
\end{tabular}
}
\end{table}

\begin{table}[h]
\caption{The MSE of autocovariance estimators of $\gamma_0=1$, $\gamma_1=\nu$ 
and $\gamma_2=0$ under the $1$-dependent error model where $\delta_{i}$'s are i.i.d. $t_{4}$ based on $500$ replications 
of size $1600.$ \label{tab.mse_plinMethods_ParksErrors2}}
\centering
\scalebox{0.45}{
\begin{tabular}{*{22}{c}}

\toprule[1.25pt]
 & \multicolumn{3}{c}{$\nu=-0.5$} & \multicolumn{3}{c}{$\nu=-0.4$} & \multicolumn{3}{c}{$\nu=-0.2$} 
 & \multicolumn{3}{c}{$\nu=0$} & \multicolumn{3}{c}{$\nu=0.2$} & \multicolumn{3}{c}{$\nu=0.4$}
 & \multicolumn{3}{c}{$\nu=0.5$} \\
 \cmidrule[1.25pt]{2-22}
 
 & $\gamma_0$ & $\gamma_1$ & $\gamma_2$ &$\gamma_0$ &  $\gamma_1$ & $\gamma_2$ &$\gamma_0$ &  $\gamma_1$ & $\gamma_2$ & $\gamma_0$ & $\gamma_1$ & $\gamma_2$ & $\gamma_0$
 & $\gamma_1$ & $\gamma_2$ & $\gamma_0$ & $\gamma_1$ & $\gamma_2$ & $\gamma_0$ & $\gamma_1$ & $\gamma_2$ \\
 
 \cmidrule[1.25pt]{2-22}
\lf{$f_1$} & \lf{0.0764} & \lf{0.0538} & \lf{0.0250} & \lf{0.0595} & \lf{0.0514} & \lf{0.0225} & \lf{0.0990} 
       & \lf{0.0457} & \lf{0.0227} & \lf{0.0659} & \lf{0.0422} & \lf{0.0218} & \lf{0.0876} & \lf{0.0459}& \lf{0.0233} & \lf{0.0754} & \lf{0.0518} &\lf{0.0237} & \lf{0.0707} & \lf{0.0467} & \lf{0.0212}\\

 \lf{$f_2$} & \lf{0.0935} & \lf{0.0521} & \lf{0.0225} & \lf{0.1013} & \lf{ 0.0568} & \lf{0.0219} & \lf{0.0868} 
       & \lf{0.0450} & \lf{0.0228} & \lf{0.0679} & \lf{0.0430} & \lf{0.0218} & \lf{0.0677} & \lf{0.0400}& \lf{0.0203}& \lf{0.0801} & \lf{0.0514}& \lf{0.0253} &\lf{0.0786} &\lf{0.0539} &\lf{0.0230}\\

 \lf{$f_3$} & \lf{0.1088} & \lf{0.0636} & \lf{0.0247} & \lf{0.0645} & \lf{0.0528} & \lf{0.0238} & \lf{0.1059} 
            & \lf{0.0448} & \lf{0.0234} & \lf{0.0877} & \lf{0.0446} & \lf{0.0228} & \lf{0.0789} & \lf{0.0491} & \lf{0.0245} & \lf{0.0964} & \lf{0.0555} & \lf{0.0239} & \lf{0.0818} &             \lf{0.0575} & \lf{0.0257} \\
 \bottomrule[1.25pt]  
\end{tabular}
}
\end{table}

\begin{table}[h]
\caption{The MSE of autocovariance estimators of $\gamma_0=1$, $\gamma_1=\nu$ 
and $\gamma_2=0$ under the $1$-dependent error model where $\delta_{i}$'s are i.i.d. $\mathcal{N}({0},1)$ based on 
$500$ replications of size $3000.$\label{tab.mse_plinMethods_ParksErrors1i}}
\centering
\scalebox{0.45}{
\begin{tabular}{*{22}{c}}

\toprule[1.25pt]
 & \multicolumn{3}{c}{$\nu=-0.5$} & \multicolumn{3}{c}{$\nu=-0.4$} & \multicolumn{3}{c}{$\nu=-0.2$} 
 & \multicolumn{3}{c}{$\nu=0$} & \multicolumn{3}{c}{$\nu=0.2$} & \multicolumn{3}{c}{$\nu=0.4$}
 & \multicolumn{3}{c}{$\nu=0.5$} \\
 \cmidrule[1.25pt]{2-22}
 
 & $\gamma_0$ & $\gamma_1$ & $\gamma_2$ &$\gamma_0$ &  $\gamma_1$ & $\gamma_2$ &$\gamma_0$ &  $\gamma_1$ & $\gamma_2$ & $\gamma_0$ & $\gamma_1$ & $\gamma_2$ & $\gamma_0$
 & $\gamma_1$ & $\gamma_2$ & $\gamma_0$ & $\gamma_1$ & $\gamma_2$ & $\gamma_0$ & $\gamma_1$ & $\gamma_2$ \\
 
 \cmidrule[1.25pt]{2-22}
\lf{$f_1$} & \lf{0.0117} & \lf{0.0116} & \lf{0.0055} & \lf{0.0120} & \lf{ 0.0111} & \lf{0.0055} & \lf{0.0120} 
       & \lf{0.0111} & \lf{0.0053} & \lf{0.0112} & \lf{0.0112} & \lf{0.0055} & \lf{0.0116} & \lf{0.0114}& \lf{0.0054} & \lf{0.0113} & \lf{0.0114} &\lf{0.0057} & \lf{0.0119} & \lf{0.0116} & \lf{0.0056}\\

 \lf{$f_2$} & \lf{0.0124} & \lf{0.0116} & \lf{0.0058} & \lf{0.0119} & \lf{0.0114} & \lf{0.0054}  &\lf{0.0120} & \lf{0.0112} & \lf{0.0053} & \lf{0.0115} & \lf{ 0.0114} & \lf{0.0057} & \lf{0.0116} & \lf{0.0112} & \lf{0.0052}& \lf{0.0122}& \lf{0.0116} & \lf{0.0055}& \lf{0.0121} &\lf{0.0115} &\lf{0.0054} \\

 \lf{$f_3$} & \lf{0.0122} & \lf{0.0120} & \lf{0.0055} & \lf{0.0125} & \lf{0.0119} & \lf{0.0059} & \lf{ 0.0119} 
            & \lf{0.0117} & \lf{0.0058} & \lf{0.0121} & \lf{0.0118} & \lf{0.0057} & \lf{0.0123} & \lf{0.0123} & \lf{0.0060} & \lf{ 0.0122} & \lf{ 0.0123} & \lf{0.0060} & \lf{0.0127} &             \lf{0.0124} & \lf{0.0058} \\
 \bottomrule[1.25pt]  
\end{tabular}
}
\end{table}

\begin{table}[h]
\caption{The MSE of autocovariance estimators of $\gamma_{0}=1$, $\gamma_{1}=\nu$ 
and $\gamma_2=0$ under the $1$-dependent error model where $\delta_{i}$'s are i.i.d. $\mathcal{N}({0},1)$ based on $500$ replications 
of size $500.$~\label{tab.mse_plinMethods_ParksErrors_1iii}}
\centering
\scalebox{0.45}{
\begin{tabular}{*{22}{c}}

\toprule[1.25pt]
 & \multicolumn{3}{c}{$\nu=-0.5$} & \multicolumn{3}{c}{$\nu=-0.4$} & \multicolumn{3}{c}{$\nu=-0.2$} 
 & \multicolumn{3}{c}{$\nu=0$} & \multicolumn{3}{c}{$\nu=0.2$} & \multicolumn{3}{c}{$\nu=0.4$}
 & \multicolumn{3}{c}{$\nu=0.5$} \\
 \cmidrule[1.25pt]{2-22}
 
 & $\gamma_0$ & $\gamma_1$ & $\gamma_2$ &$\gamma_0$ &  $\gamma_1$ & $\gamma_2$ &$\gamma_0$ &  $\gamma_1$ & $\gamma_2$ & $\gamma_0$ & $\gamma_1$ & $\gamma_2$ & $\gamma_0$
 & $\gamma_1$ & $\gamma_2$ & $\gamma_0$ & $\gamma_1$ & $\gamma_2$ & $\gamma_0$ & $\gamma_1$ & $\gamma_2$ \\
 
 \cmidrule[1.25pt]{2-22}
\lf{$f_1$} & \lf{0.3974} & \lf{0.3901} & \lf{0.1764} & \lf{0.3843} & \lf{0.3876} & \lf{0.1729} & \lf{ 0.3859} 
       & \lf{0.3986} & \lf{0.1758} & \lf{0.3867} & \lf{0.3880} & \lf{0.1743} & \lf{0.4003} & \lf{ 0.3976}& \lf{0.1778} & \lf{0.3962} & \lf{0.3940} &\lf{0.1768} & \lf{0.4035} & \lf{0.4062} & \lf{0.1848}\\

 \lf{$f_2$} & \lf{0.3970} & \lf{0.3895} & \lf{0.1760} & \lf{0.3861} & \lf{0.3957} & \lf{ 0.1807}  &\lf{0.3918} & \lf{0.3908} & \lf{0.1823} & \lf{0.3907} & \lf{0.3930 } & \lf{0.1781} & \lf{0.3938} & \lf{0.3990} & \lf{0.1838}& \lf{ 0.3993}& \lf{0.3954} & \lf{0.1763}& \lf{0.3910} &\lf{0.3838} &\lf{0.1709} \\

 \lf{$f_3$} & \lf{0.4587} & \lf{0.5176} & \lf{0.2441} & \lf{0.4648} & \lf{0.5125} & \lf{0.2466} & \lf{0.4681} 
            & \lf{0.5192} & \lf{0.2461} & \lf{0.4627} & \lf{0.5206} & \lf{0.2475} & \lf{0.4666} & \lf{0.5312} & \lf{0.2511} & \lf{0.4771} & \lf{ 0.5159} & \lf{0.2401} & \lf{0.4629} &             \lf{0.5205} & \lf{0.2493} \\
 \bottomrule[1.25pt]  
\end{tabular}
}
\end{table}

\begin{table}[h]
\caption{The MSE of autocovariance estimators of $\gamma_{0}=1$, $\gamma_{1}=\nu$ 
and $\gamma_2=0$ under the $1$-dependent error model where $\delta_{i}$'s are i.i.d. $\mathcal{N}({0},1)$ based on 
$500$ replications of size $1000.$
\label{tab.mse_plinMethods_ParksErrors1ii}}
\centering
\scalebox{0.45}{
\begin{tabular}{*{22}{c}}

\toprule[1.25pt]
 & \multicolumn{3}{c}{$\nu=-0.5$} & \multicolumn{3}{c}{$\nu=-0.4$} & \multicolumn{3}{c}{$\nu=-0.2$} 
 & \multicolumn{3}{c}{$\nu=0$} & \multicolumn{3}{c}{$\nu=0.2$} & \multicolumn{3}{c}{$\nu=0.4$}
 & \multicolumn{3}{c}{$\nu=0.5$} \\
 \cmidrule[1.25pt]{2-22}
 
 & $\gamma_0$ & $\gamma_1$ & $\gamma_2$ &$\gamma_0$ &  $\gamma_1$ & $\gamma_2$ &$\gamma_0$ &  $\gamma_1$ & $\gamma_2$ & $\gamma_0$ & $\gamma_1$ & $\gamma_2$ & $\gamma_0$
 & $\gamma_1$ & $\gamma_2$ & $\gamma_0$ & $\gamma_1$ & $\gamma_2$ & $\gamma_0$ & $\gamma_1$ & $\gamma_2$ \\
 
 \cmidrule[1.25pt]{2-22}
\lf{$f_1$} & \lf{0.0997} & \lf{ 0.0966} & \lf{0.0452} & \lf{0.0994} & \lf{0.0979} & \lf{0.0450} & \lf{0.0971} 
       & \lf{0.0994} & \lf{0.0451} & \lf{0.0980} & \lf{0.0990} & \lf{0.0457} & \lf{0.0975} & \lf{0.0965}& \lf{0.0449} & \lf{0.1012} & \lf{0.1002} &\lf{0.0454} & \lf{0.1016} & \lf{0.1025} & \lf{0.0477}\\

 \lf{$f_2$} & \lf{0.0948} & \lf{0.1017} & \lf{0.0456} & \lf{ 0.0978} & \lf{0.0988} & \lf{0.0456}  &\lf{0.0957} & \lf{0.0987} & \lf{0.0441} & \lf{0.1000} & \lf{0.0988} & \lf{0.0433} & \lf{0.1034} & \lf{0.1003} & \lf{0.0454}& \lf{0.1016}& \lf{0.1005} & \lf{0.0461}& \lf{0.1018} &\lf{0.1027} &\lf{0.0477} \\

 \lf{$f_3$} & \lf{0.1050} & \lf{0.1142} & \lf{0.0535} & \lf{ 0.1101} & \lf{0.1118} & \lf{0.0541} & \lf{0.1065} 
            & \lf{0.1113} & \lf{0.0538} & \lf{0.1075} & \lf{0.1151} & \lf{0.0518} & \lf{0.1073} & \lf{0.1139} & \lf{0.0538} & \lf{0.1084} & \lf{0.1144} & \lf{0.0536} & \lf{0.1090} &             \lf{0.1153} & \lf{0.0544} \\
 \bottomrule[1.25pt]  
\end{tabular}
}
\end{table}

\subsection{Autocovariance estimation in a model with AR($1$) errors}
We also conducted an additional experiment with the aim of comparing our method to a possible competitor. In particular, we considered a situation where the errors $\eps_{i}$ are generated by an AR($1$) process $\eps_{i}=\phi\eps_{i-1}+\zeta_{i}$ where $\zeta_{i}\sim \mathcal{N}({0},1)$ and independent. The method proposed in our manuscript is designed for the case where model errors are generated by an $m$-dependent process with a finite $m;$ however, for a causal AR(1) process with $0<\phi<1$ the autocovariance at lag $k$ is proportional to $\phi^{|k|},$ thus decreasing at an exponential rate. Due to this, we hypothesize that our method may perform reasonably well if the errors are generated by an AR($1$) process with $\phi$ that is not too close to $1$ in absolute value. More specifically, we selected values of $\phi$ ranging from $0.1$ to $0.5$ with a step of $0.10.$ As before, we compare the performance of our method when estimating variance and
autocovariance at lags 1 and 2 of the process $\{\eps_{i}\}$ to that of 
\cite{Hall.VanKeilegom.03} and \cite{Chakar.etal.16}.
The method of \cite{Chakar.etal.16} assumes a mean function with a finite number of breakpoints and was designed to estimate the coefficient $\phi$ only and it does not provide a direct estimate of the autocovariance. Therefore, we proceed as follows. Let $\hat\phi$ be the robust estimator of $\phi$ obtained using the method of \cite{Chakar.etal.16}. Then, we compute the estimated variance and autocovariance at lags $1$ and $2$ for this method as $\frac{1}{1-\hat\phi^2},$ $\frac{\hat\phi}{1-\hat\phi^2}$, and $\frac{\hat\phi^{2}}{1-\hat\phi^2}$, respectively.  

In this experiment, the three choices of the smooth function $f(x)$ remain the same as in the previous one. When applying our estimator to an AR(1) error process, we choose the value $m=2.$ The method of \cite{Hall.VanKeilegom.03} method is used with the choices of smoothing parameters recommended in Section $3$ of their paper: the first parameter $m_{1}=n^{0.4}$ and the second parameter $m_{2}=\sqrt{n}.$ {We used the method of \cite{Chakar.etal.16}  
with the default arguments of the R function \texttt{AR1seg\_func} from the package \texttt{AR1seq} version 1.0.}
The results of the method comparison is given in Tables 
\ref{tab.mse_plinMethods_AR1Errors}-\ref{tab.mse_HallvanKeilegomMethods_AR1Errors}-\ref{tab.mse_Chakar_AR1Errors}. 

Comparing the performance of our approach to that of \cite{Hall.VanKeilegom.03} we note that, perhaps unsurprisingly, our method provides a universally better performance. In most cases, the MSEs of all estimators for our method are smaller by an order of magnitude compared to that of the method of \cite{Hall.VanKeilegom.03}. That includes all three choices of the smooth function $f(x)$ and all possible choices of the autoregressive coefficient $\phi.$ This suggests that the method of \cite{Hall.VanKeilegom.03} is not at all resistant to the presence of breakpoints in the mean function of the regression model. Comparing the MSEs shown in Table \ref{tab.mse_plinMethods_AR1Errors} to those exhibited by our estimator under the assumption of $m$-dependency for the same sample size, it appears that our method can also handle moderate departures from the assumption of $m$-dependency in the error process of the model rather well. 

\begin{table}[h]
\caption{The MSE of autocovariance estimators at lags $0$, $1$ and $2$ obtained using our approach 
under the AR(1)-dependent error model where the coefficient $\phi$ ranges between $0.1$ and $0.5.$ The experiment is based on $500$ replications of size $1600.$\label{tab.mse_plinMethods_AR1Errors}}
\centering
\scalebox{0.45}{
\begin{tabular}{*{16}{c}}
\toprule[1.25pt]

 & \multicolumn{3}{c}{$\phi=0.10$} & \multicolumn{3}{c}{$\phi=0.20$} & \multicolumn{3}{c}{$\phi=0.30$} & \multicolumn{3}{c}{$\phi=0.40$} & \multicolumn{3}{c}{$\phi=0.50$} \\
 \cmidrule[1.25pt]{2-16}
 
 & $\gamma_0$ & $\gamma_1$ & $\gamma_2$ & $\gamma_0$ & $\gamma_1$ & $\gamma_2$ & $\gamma_0$ & $\gamma_1$ & $\gamma_2$ & $\gamma_0$ & $\gamma_1$ & $\gamma_2$ & $\gamma_0$ & $\gamma_1$ & $\gamma_2$  \\
 
 \cmidrule[1.25pt]{2-16}
 \LF{$f_1$} & \LF{0.0379} & \LF{0.0377} & \LF{0.0180} & \LF{0.0346} & \LF{0.0376} & \LF{0.0177} 
            & \LF{0.0289} & \LF{0.0356} & \LF{0.0159} & \LF{0.0178} & \LF{0.0307} & \LF{0.0127} & \LF{0.0046} & \LF{0.0158} & \LF{0.0044} \\

 \LF{$f_2$} & \LF{0.0385} & \LF{0.0375} & \LF{0.0174} & \LF{0.0367} & \LF{0.0385} & \LF{0.0178} 
            & \LF{0.0288} & \LF{0.0339} & \LF{0.0150} & \LF{0.0168} & \LF{0.0298} & \LF{0.0122} & \LF{0.0040} & \LF{0.0156} & \LF{0.0045}  \\
 
 \LF{$f_3$} & \LF{0.0402} & \LF{0.0429} & \LF{0.0204} & \LF{0.0376} & \LF{0.0413} & \LF{0.0193}           & \LF{0.0304} & \LF{0.0388} & \LF{0.0179} & \LF{0.0181} & \LF{ 0.0317} & \LF{0.0132} & \LF{0.0051} & \LF{0.0185} & \LF{0.0052} \\
 \bottomrule[1.25pt]  
\end{tabular}
}
\end{table}

\begin{table}[h]
\caption{The MSE of autocovariance estimators at lags $0$, $1$ and $2$ obtained using the approach of 
\cite{Hall.VanKeilegom.03} under the AR(1)-dependent error model where the coefficient $\phi$ ranges between $0.1$ and $0.5.$ The experiment is based on $500$ replications of size $1600.$\label{tab.mse_HallvanKeilegomMethods_AR1Errors}}
\centering
\scalebox{0.45}{
\begin{tabular}{*{16}{c}}

\toprule[1.25pt]

 & \multicolumn{3}{c}{$\phi=0.10$} & \multicolumn{3}{c}{$\phi=0.20$} & \multicolumn{3}{c}{$\phi=0.30$} 
 & \multicolumn{3}{c}{$\phi=0.40$} & \multicolumn{3}{c}{$\phi=0.50$} \\
 \cmidrule[1.25pt]{2-16}
 
 & $\gamma_0$ & $\gamma_1$ & $\gamma_2$ & $\gamma_0$ & $\gamma_1$ & $\gamma_2$ & $\gamma_0$ & $\gamma_1$
 & $\gamma_2$ & $\gamma_0$ & $\gamma_1$ & $\gamma_2$ & $\gamma_0$ & $\gamma_1$ & $\gamma_2$   \\
 
 \cmidrule[1.25pt]{2-16}
 \LF{$f_1$} & \LF{3.1440} & \LF{ 2.9200} & \LF{2.7421} & \LF{3.1255} & \LF{2.9103} & \LF{2.8207} 
            & \LF{3.1574} & \LF{2.9350} & \LF{3.0245} & \LF{3.1482} & \LF{2.9223} & \LF{3.3461} & \LF{3.0916} & \LF{2.8666} & \LF{3.9038} \\

 \LF{$f_2$} & \LF{3.1411} & \LF{2.9141} & \LF{2.7291} & \LF{3.1225} & \LF{2.9018} & \LF{2.8076} 
            & \LF{ 3.1152} & \LF{2.8872} & \LF{2.9750} & \LF{3.1342} & \LF{2.9077} & \LF{3.3323} & \LF{3.0804} & \LF{2.8685} & \LF{3.9204}   \\
 
 \LF{$f_3$} & \LF{5.0621} & \LF{ 4.7845} & \LF{4.5481} & \LF{5.0719} & \LF{4.7850} & \LF{4.6989} 
            & \LF{5.0394} & \LF{4.7565} & \LF{4.9638} & \LF{5.0202} & \LF{4.7317} & \LF{5.4616} & \LF{4.9952} & \LF{4.7178} & \LF{6.3817}  \\
 \bottomrule[1.25pt]  
\end{tabular}
}
\end{table}

At the same time, it appears that our method can, at least partially, holds its own against {the method of \cite{Chakar.etal.16} that, like ours, has been devised to account for a possibility of a finite number of breaks in the mean}. This method exhibits better performance than our method for relatively small values of $\phi$ such as $\phi=0.1$ and $\phi=0.2$. This better performance is observed for all choices of the function $f$. Its performance deteriorates significantly, however, for larger values of $\phi$. As a matter of fact, some of the MSEs of estimators obtained by this method for $\phi=0.5$ are an order or two orders of magnitude larger than those of our estimators under the same circumstances. This suggests that, in some cases, our approach can handle autocovariance estimation when the mean has a number of breakpoints better than the robust method of \cite{Chakar.etal.16}. Moreover, as opposed to our ACF estimation method, Chakar's is very slow. 
For example, the results of the Table \eqref{tab.mse_Chakar_AR1Errors} were obtained using $534.89$ sec while those in Table ~\ref{tab.mse_plinMethods_ParksErrors2} took only about $1.80$ sec to complete. All calculations were performed on a Dell Latitude $5330$ laptop with $16$ GB RAM and a $12$th Gen Intel(R) Core(TM) i$5-1235$U  $1.30$ GHz processor.
The reason for having such a long execution time in the estimation of the parameter $\phi$
is that this estimate is returned only when the method has estimated the entire change point
model, which is known to be time-consuming. {The user interested in Chakar's robust estimation method, can reduce the execution time considerably by using the function \texttt{dbacf\_AR1} of the
\texttt{R} package \texttt{dbacf}.}

\begin{table}[h]
\caption{The MSE of autocovariance estimators at lags $0$, $1$ and $2$ obtained using the approach of 
\cite{Chakar.etal.16} under the AR(1)-dependent error model where the coefficient $\phi$ ranges 
between $0.1$ and $0.5.$ The experiment is based on $500$ replications of size $1600.$\label{tab.mse_Chakar_AR1Errors}}
\centering
\scalebox{0.45}{
\begin{tabular}{*{16}{c}}

\toprule[1.25pt]

 & \multicolumn{3}{c}{$\phi=0.10$} & \multicolumn{3}{c}{$\phi=0.20$} & \multicolumn{3}{c}{$\phi=0.30$} 
 & \multicolumn{3}{c}{$\phi=0.40$} & \multicolumn{3}{c}{$\phi=0.50$} \\
 \cmidrule[1.25pt]{2-16}
 
 & $\gamma_0$ & $\gamma_1$ & $\gamma_2$ & $\gamma_0$ & $\gamma_1$ & $\gamma_2$ & $\gamma_0$ & $\gamma_1$
 & $\gamma_2$ & $\gamma_0$ & $\gamma_1$ & $\gamma_2$ & $\gamma_0$ & $\gamma_1$ & $\gamma_2$   \\
 
 \cmidrule[1.25pt]{2-16}
 \LF{$f_1$} & \LF{0.0005} & \LF{0.0065} & \LF{0.0005} & \LF{0.0023} & \LF{0.0179} & \LF{ 0.0023} 
            & \LF{0.0078} & \LF{0.0888} & \LF{0.0078} & \LF{0.0319} & \LF{0.3981} & \LF{0.0319} & \LF{0.0713} & \LF{1.6504} & \LF{0.0713} \\

 \LF{$f_2$} & \LF{0.0006} & \LF{0.0076} & \LF{0.0006} & \LF{0.0018} & \LF{0.0187} & \LF{0.0018} 
            & \LF{0.0053} & \LF{0.0819} & \LF{0.0053} & \LF{0.0302} & \LF{0.3806} & \LF{0.0302} & \LF{0.1388} & \LF{1.6557} & \LF{0.1388}   \\
 
 \LF{$f_3$} & \LF{0.0006} & \LF{0.0079} & \LF{0.0006} & \LF{0.0021} & \LF{0.0182} & \LF{0.0021} 
            & \LF{0.0086} & \LF{0.0833} & \LF{0.0086} & \LF{0.0223} & \LF{0.3819} & \LF{0.0223} & \LF{0.0811} & \LF{1.6745} & \LF{0.0811}  \\
 \bottomrule[1.25pt]  
\end{tabular}
}
\end{table}

Finally, note that the choice of the estimator gap $m,$ when applying our estimator to the model with AR($1$) error process, is somewhat arbitrary. Our choice of the order $2$ was based on the fact that, for any difference sequence order $\ell,$ the bias of the resulting estimator is proportional to $\ell\,(m+1)$ which may play a fairly substantial role when the number of observations $n$ is not too large. To illustrate the benefits of this assumption, we also applied our method to the same model with an AR($1$) error process, the same range of values of $\phi$ and the choice of the smooth function $f_1(x)=1-2x$ for any $x\in [0,1],$ while choosing $m=3.$ The results are given in the Table \ref{tab.mse_plinMethods_AR1Errors_m3}. Note that the resulting mean squared errors are mostly larger than those obtained the choice of $m=2$ although by less than the full order of magnitude. Note, however, that even in this misspecification case our estimator outperforms \cite{Hall.VanKeilegom.03}'s uniformly and \cite{Chakar.etal.16} in a number of cases (especially for larger values of the coefficient $\phi$) when the sample size is the same. 

\begin{table}[h]
\caption{The MSE of autocovariance estimators at lags $0$, $1$ and $2$ obtained using our approach under the AR(1)-dependent error model where the coefficient $\phi$ ranges between $0.1$ and $0.5,$ $m=3$ and $f=f_1.$ The experiment is based on $500$ replications of size $1600.$\label{tab.mse_plinMethods_AR1Errors_m3}}
\centering
\scalebox{0.45}{
\begin{tabular}{*{16}{c}}
\toprule[1.25pt]

 &\multicolumn{3}{c}{$\phi=0.10$} & \multicolumn{3}{c}{$\phi=0.20$} & \multicolumn{3}{c}{$\phi=0.30$} & \multicolumn{3}{c}{$\phi=0.40$} & \multicolumn{3}{c}{$\phi=0.50$}\\
 
 & $\gamma_0$ & $\gamma_1$ & $\gamma_2$ & $\gamma_0$ & $\gamma_1$ & $\gamma_2$ & $\gamma_0$ & $\gamma_1$ & $\gamma_2$ & $\gamma_0$ & $\gamma_1$ & $\gamma_2$ & $\gamma_0$ & $\gamma_1$ & $\gamma_2$ \\
\cmidrule[1.25pt]{2-16}
 \LF{$f_1$} & \LF{0.0673} & \LF{0.0662} & \LF{ 0.0375} & \LF{0.0666} & \LF{0.0674} & \LF{0.0381} 
            & \LF{0.0673} & \LF{0.0675} & \LF{0.0380} & \LF{0.0563} & \LF{0.0654}  & \LF{0.0370} & \LF{0.0361} & \LF{0.0528} & \LF{0.0273} \\

 \bottomrule[1.25pt]  
\end{tabular}
}
\end{table}


%
\section{Application}~\label{sec.apps}
In this section we consider the dataset \texttt{gtemp\_land}
from the R package \texttt{astsa} by \cite{Stoffer.Poison.23}.
This dataset contains observations of annual temperature anomalies 
(in Celsius degree) averaged over the Earth's land area from
1880 to 2021. We present a brief analysis of this dataset autocovariance
structure.

 \begin{figure}[htb]
\centering
    \includegraphics[width=\linewidth]{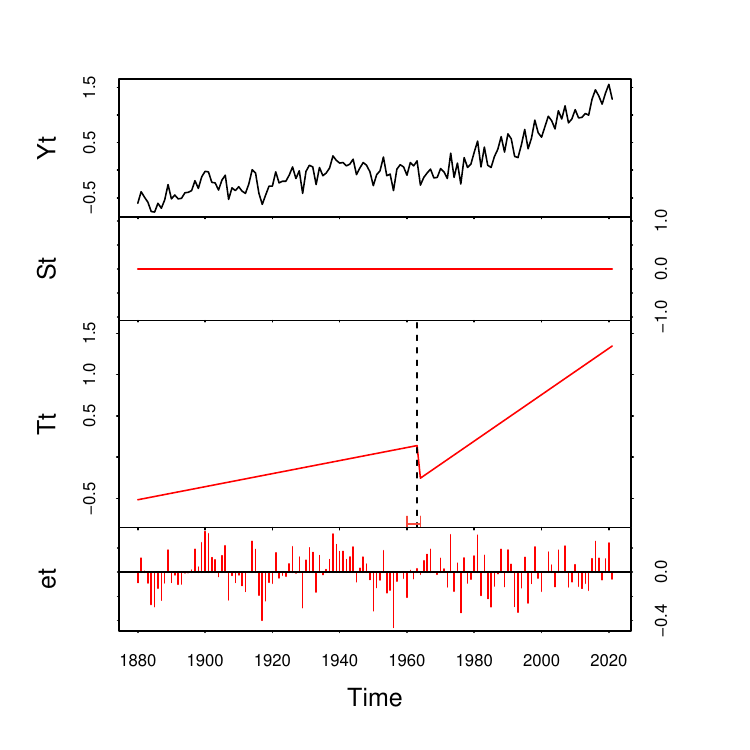} 
  \caption{\footnotesize{Change point analysis of \texttt{gtemp\_land} dataset.}}
  \label{fig:app} 
\end{figure}

Figure \ref{fig:app} shows the \texttt{gtemp\_land} (top panel) with a $95\%$ statistically
significant breakpoint (at 1963) in its estimated trend (third panel from top to bottom);
the second and fourth panels (from top to bottom) show the estimated seasonal
component and the residuals, respectively. This breakpoint was estimated using the R package \texttt{bfast}
by \cite{Verbesselt.10}, with no seasonal argument and using default
values for the remaining arguments of the procedure.

\cite{Verbesselt.10}'s allows for the estimation of unknown change points
in the trend and seasonal component of time series. Statistically,
this method is based on \cite{Chu.etal.1995}'s MOSUM test for no changes 
in the parameters of a sequence of local linear regressions, see also \cite{Zeileis.etal.2002} for further extensions to the method and its implementation.
Computationally, \texttt{bfast} is an iterative algorithm allowing for 
fast computations due to the use of dynamic programming to keep track 
of both the number and positions of the estimated change points. 
The total number of change points, unknown a priori, is estimated through 
optimizing the BIC.

Due to the above, we believe that it is pertinent to analyze
these observations with our approach. Apparently, \cite{Hall.VanKeilegom.03}
analyzed this dataset with observations until 1985 and
assumed AR(1) and AR(2) error processes.
Like us, these authors did not consider a seasonal component. 
We did not consider a seasonal component because these are annually collected observations with no evident
bi-, five-, etc., annual structure. In addition, a smooth seasonal component 
can be construed as a part of our smooth function $f(\cdot)$.
Table \ref{tab:app} shows \cite{Hall.VanKeilegom.03}'s ACF estimators as 
well as the bias-reducing, second-order, $(m+1)$-gapped, autocovariance 
estimator of \cite{Tecuapetla.Munk.2017}, and our estimator. For the latter 
two estimators, we assumed that $m=1,2,3,4$.

\begin{table}[ht]
\centering
\caption{\footnotesize ACF estimation for the annual average temperature anomalies of Earth's land area from 1880 to 2021.}~
\label{tab:app}
\fontsize{7}{9}\selectfont
\begin{tabular}[t]{c|c|cccccccccccc|c}
\toprule
\hline
&&
\multicolumn{2}{c}{$\gamma_{HV}$} 
&&
\multicolumn{4}{c}{$\gamma_{TM}$}
&&
\multicolumn{4}{c|}{$\gamma_{plin}$} &\\
\hline
&& $AR(1)$ & $AR(2)$ && $m=1$ & $m=2$ & $m=3$ & $m=4$ && $m=1$ & $m=2$ & $m=3$ & $m=4$ &\\
\cline{3-4}  \cline{6-9}  \cline{11-14}
&$\wh{\gamma}_0$ & 0.035 & 0.035 
		   && 0.025 & 0.028 & 0.023 & 0.027 
		   && 0.025 & 0.029 & 0.026 & 0.031 &\\
&$\wh{\gamma}_1$ & 0.014 & 0.014 
		   && 0.004 & 0.007 & 0.002 & 0.006 
		   && 0.008 & 0.005 & 0.010 & 0.011 &\\
&$\wh{\gamma}_2$ & & 0.010 
		   &&  & 0.003 & -0.002 & 0.002 
		   &&  & 0.001 & 0.006 & 0.007 &\\
&$\wh{\gamma}_3$ & &  
		   &&  &  & -0.005 & -0.002 
		   &&  &  & 0.002 & 0.004 &\\
&$\wh{\gamma}_4$ & &  
		   &&  &  &  & 0.003 
		   &&  &  &  & 0.007 &\\
\hline
\bottomrule
\end{tabular}
\vspace{0.25cm}
\caption*{\footnotesize Notation:  ${\gamma}_{HV}$, ${\gamma}_{TM}$, 
${\gamma}_{plin}$, \cite{Hall.VanKeilegom.03}'s, \cite{Tecuapetla.Munk.2017}'s, 
and this paper ACF estimates, respectively.}
\end{table}

Although all estimates for $\gamma_0$ are of the same order of magnitude,
\cite{Hall.VanKeilegom.03}'s is slightly larger than the estimates obtained 
with the other two methods; the same can be said for the estimates of $\gamma_1$.
This might be attributed to the fact that Hall and Van Keilegom's method
does not take into account the apparent breakpoint within the time series. 
It seems that independently of the value of $m$, the bias-reducing,
difference-based estimators shown in this section provide similar results. 


%
%
%

\section*{Supplementary Materials for ``Autocovariance estimation via difference
schemes for a semiparametric change point model with $m$-dependent errors'' by
Michael Levine and Inder Tecuapetla-G\'omez}

\appendix

\section[Appendix A\hfill]{Proof of Theorems~\ref{theo.DBE-partial}, \ref{eq.root-n}, \ref{th_var}, \ref{theo_var_GAMMAh}}~\label{sec.supp}
In this section we will employ the following notation.
For $x\in \R$, $\lfloor x\rfloor = \max\{y\in \Z\mid y\leq x\}$. 
For $0\leq i\leq K$, $t_i=\lfloor n\,\tau_i\rfloor$, where $\tau_i$ is a change point.
For $a,b\in \R$, $a \wedge b = \min (a,b)$.
For the subset $\mc A$, $\ind_{\mc A}$ denotes the indicator function on $\mc A$.
We will write $n_{\ell_{m+1}}$ to denote $n-\ell(m+1)$ and $n_h=n-h$.
For $r>0$, $\sum_{i,j}^{(r)}=\sum_{i=1}^{r-1}\,\sum_{j=i+1}^{r}$.

Recall that any continuous function $f$ defined on a compact set $\mc{C}$ is also uniformly
continuous. That is, there exist an increasing function $\omega(\cdot)$, called
the modulus of continuity of $f$, that is vanishing
at zero and continuous at zero, such that $|f(x)-f(y)|\leq \omega(|x-y|)$ for all
$x,y\in \mc{C}$.

%
\begin{pf}[Proof of Theorem~\ref{theo.DBE-partial}]
Let $\bs{0}_m$ denote an $m\times 1$ vector of zeros (this is a vector with $m$ rows and 1 column). 
This allows us to define the $(\ell(m+1)+1)$-dimensional row vector 
$\bs{w}_\ell^\top 
= 
( d_0 \: \bs{0}_m^{\top} \: d_1 \: \bs{0}_m^{\top} \: \cdots \: d_{\ell-1} \: \bs{0}_m^{\top} \: d_\ell )$,
the corresponding $2\times (\ell(m+1)+2)$ matrix
\[
    \tilde{D} 
    =
	\begin{pmatrix} d_0 & \boldsymbol{0}_m^\top & d_1 & \boldsymbol{0}_m^\top 
	                    & \cdots & d_{\ell-1} & \boldsymbol{0}_m^\top & d_\ell & 0\\ 
	0 & d_0 & \boldsymbol{0}_m^\top & d_1 & \boldsymbol{0}_m^\top & \cdots & d_{\ell-1} 
	& \boldsymbol{0}_m^\top & d_\ell		
	\end{pmatrix} 
	=        
    \begin{pmatrix}
        \bs{w}_\ell^\top & 0\\
          0 & \bs{w}_\ell^\top
    \end{pmatrix},
\]
and the $(\ell(m+1)+2) \times (\ell(m+1) +2)$ symmetric matrix $D = \Tilde{D}^\top\,\tilde{D}$.

It is not difficult to see that for $i\leq n_{\ell_{m+1}}-1$, 
\[
    \bs{y}_{i:(i+1+\ell(m+1))}^\top\,D\,\bs{y}_{i:(i+1+\ell(m+1))}
    =
    \sum_{j=i}^{i+1}\,\Delta_{\ell,m+1}^2(y_j).
\] 
Therefore,
\[
	2n_{\ell_{m+1}}p(\bs{d})Q_{\ell,m+1}(\bs{y};\bs{d})
    =
    \Delta_{\ell,m+1}^2(y_1) + \Delta_{\ell,m+1}^2(y_{n_{\ell_{m+1}}})
    + 
    \sum_{i=1}^{n_{\ell_{m+1}}-1}\,\bs{y}_{i:(i+1+\ell(m+1))}^\top\,D\,\bs{y}_{i:(i+1+\ell(m+1))}.
\]

Since
\begin{align*}
	\bs{y}_{i:(i+1+\ell(m+1))}^\top\,D\,\bs{y}_{i:(i+1+\ell(m+1))}
	&=
	\bs{f}_{i:(i+1+\ell(m+1))}^\top\,D\,\bs{f}_{i:(i+1+\ell(m+1))}\\
	&+
	2\,\bs{f}_{i:(i+1+\ell(m+1))}^\top\,D\,\bs{g}_{i:(i+1+\ell(m+1))}\\
	&+
	\bs{g}_{i:(i+1+\ell(m+1))}^\top\,D\,\bs{g}_{i:(i+1+\ell(m+1))}\\
	&+
	2\,\bs{f}_{i:(i+1+\ell(m+1))}^\top\,D\,\bs{\veps}_{i:(i+1+\ell(m+1))}\\
	&+
	2\,\bs{g}_{i:(i+1+\ell(m+1))}^\top\,D\,\bs{\veps}_{i:(i+1+\ell(m+1))}\\
	&+
	\bs{\veps}_{i:(i+1+\ell(m+1))}^\top\,D\,\bs{\veps}_{i:(i+1+\ell(m+1))},
\end{align*}
and because $(\veps_{i})$ is a zero mean process, we have for any 
$i\leq j$ that for any $(j-i+1)\times(j-i+1)$ matrix $\Sigma$, 
$\ME[ \bs{\veps}_{i:j}^\top\,\Sigma\,\bs{\veps}_{i:j} ] = \mbox{tr}(\Sigma\,\VAR( \bs{\veps}_{i:j} ))$ 
see e.g.~\cite{provost1992quadratic} p.~51, it turns out that
\begin{align*}
	\ME \bs{Q}_{\ell,m+1}(\bs{y};\bs{w})
	&=
	\frac{1}{2\,n_{\ell_{m+1}}}\,\ME\left[\Delta_{\ell,m+1}^2(y_1)+\Delta_{\ell,m+1}^2(y_{n-\ell(m+1)})\right]\\
	&+
	\frac{1}{2\,n_{\ell_{m+1}}}\,\sum_{i=1}^{n_{\ell_{m+1}}-1}\,
	\bs{f}_{i:(i+1+\ell(m+1))}^\top\,D\,\bs{f}_{i:(i+1+\ell(m+1))}\\
	&+
	\frac{1}{n_{\ell_{m+1}}}\,\sum_{i=1}^{n_{\ell_{m+1}}-1}\,
	\bs{f}_{i:(i+1+\ell(m+1))}^\top\,D\,\bs{g}_{i:(i+1+\ell(m+1))}\\
	&+
	\frac{1}{2\,n_{\ell_{m+1}}}\,\sum_{i=1}^{n_{\ell_{m+1}}-1}\,
	\bs{g}_{i:(i+1+\ell(m+1))}^\top\,D\,\bs{g}_{i:(i+1+\ell(m+1))}\\
	&+
	\frac{1}{2\,n_{\ell_{m+1}}}\,\sum_{i=1}^{n_{\ell_{m+1}}-1}\,
	\mbox{tr}(D\,\VAR( \bs{\veps}_{i:(i+1+\ell(m+1))} ))\\
	&=
	\frac{1}{2\,n_{\ell_{m+1}}}\,\left[ 2\,\gamma_0 + o\left( \omega^2(\frac{m+1}{n}) \right) \right]\qquad 
	(\mbox{due to Lemma~\eqref{lemma0}})\\
	&+
	o \left( \omega^2 \left( \frac{m+1}{n} \right) \right)\qquad (\mbox{due to Lemma~\eqref{lemma}})\\
	&+
	o \left( \omega\left(\frac{m+1}{n}\right)\,(m+1)R_\ell(\bs{d})\frac{J_K^{||}}{n_{\ell_{m+1}}} \right)
	\qquad (\mbox{due to Lemma~\ref{lemma4} })\\
	&+
	o\left( (m+1)P_\ell(\bs{d}) \frac{J_K}{n_{\ell_{m+1}}} \right)\qquad (\mbox{due to Lemma~\eqref{lemmaBIAS1}})\\
	&+
	\frac{n_{\ell_{m+1}}-1}{2\,n_{\ell_{m+1}}}(2\gamma_0). 
\end{align*}
The latter identity follows from Proposition~A.1 in the Appendix of \cite{Tecuapetla.Munk.2017}.
The result follows after some algebra. 
\end{pf}\medskip

\begin{pf}[Proof of Theorem~\ref{eq.root-n}]
	Let $M_{K_n}=\max_{0\leq j\leq K_n} |a_{j+1} - a_j|$. First note that
\[
	J_{K_n}^{||}
	=
	\sum_{j=0}^{K_n}|a_j-a_{j+1}|\leq K_n\,M_{K_n}=o(1).
\]	
	Then, due to Remark~\ref{rm1} and the assumption on the largest jump we deduce that
    \[
        J_{K_n} \leq K_n\,M_{K_n}^2 = o (n^{-\epsilon}).
    \]    
        
	Next, observe that
\[
	H_{K_n}^{||}
	\leq 
	M_{K_n}\,n\,\sum_{j=1}^{K_n}\left| \frac{t_j}{n} - \frac{m+1}{2n} \right|
	\leq
	2n\,M_{K_n}\,K_n
	=
	\mc O(n)
\]

	Applying Theorem~\ref{th_var} and \ref{theo_var_GAMMAh}
	along with the expression derived above, we get
\begin{align*}
	\BIAS ( \sqrt{n} \, \wh{\gamma}_h ) 
	&=
	\sqrt{n}\omega^2\left( \frac{m+1}{n} \right) + \mc O(\omega\left(\frac{m+1}{n}\right)\frac{J_{K_n}^{||}}{\sqrt{n}})
	+
	\mc O\left(\frac{J_K}{\sqrt{n}}\right)
	=
	o(1)\\
	\VAR(\sqrt{n}\wh{\gamma}_h)
	&=
	\vartheta_3(m,h,\gamma(\cdot))
	+
	\left(2(m+h)+6 + o(n^{-1})\right)o(n^{-1/2})\\
	&+
	\left( 2(m+1)\left\{ 1 + \gamma_h -\gamma_0 + o(n^{-1}) \right\}  \right)\,J_{K_n}
	+
	\left( o(n^{-1}) - 1 \right)\,\omega\left(\frac{m+1}{n}\right) J_{K_n}^{||}\\
	&+
	\omega\left( \frac{m+1}{n} \right)\mc O\left( \frac{H_{K_n}^{||}}{n} \right)
	=
	\mc O\left( 1 \right).
\end{align*}    
	This complete the proof.
\end{pf}

\begin{pf}[Proof of Theorem~\ref{th_var}]
We will write $\Delta_{m+1}$ instead of $\Delta_{1,m+1}$.
Writing 
$\wh{\gamma}_0=\bs{Q}_{1,m+1}(\bs{y};\bs{d}_1) = n_{m+1}^{-1}\, \sum_{i=1}^{n_{m+1}}\,\Delta_{m+1}^2(y_i; \bs{d}_1)$, 
we get that
\begin{align}~\label{eq.variance}
    \VAR( \wh{\gamma}_0 )  
    =
	n_{m+1}^{-2}\left[ \sum_{i=1}^{n_{m+1}}\,\VAR( \Delta_{m+1}^2(y_i;\bs{d}_1) )
    +
    2\,\sum_{i,j}^{n_{m+1}}\,\COV( \Delta_{m+1}^2(y_i;\bs{d}_1), \Delta_{m+1}^2(y_j;\bs{d}_1) ) \right].
\end{align}

Note that we can write 
$\Delta_{m+1}(y_i; \bs{d}_1) = \Delta_{m+1}(f_i + g_i; \bs{d}_1) + \Delta_{m+1}(\varepsilon_i; \bs{d}_1)$. 
In order to make the notation easier, we will suppress the vector $\bs{d}_1$ from this point on hoping
for this not to cause any confusion. 

By direct calculation, we get that 
\begin{align}
    \VAR( \Delta_{m+1}^2(y_i) ) 
    &= 
    \VAR( \Delta_{m+1}^2(f_i + g_i) 
    + 
    2\Delta_{m+1}(f_i + g_i) \, \Delta_{m+1}(\varepsilon_i) 
    + 
    \Delta_{m+1}^2(\varepsilon_i) )\notag\\
    &=
    4\Delta_{m+1}^2(f_i + g_i)\,\VAR( \Delta_{m+1}(\varepsilon_i) ) 
    + 
    \VAR( \Delta_{m+1}^2(\varepsilon_i) )\notag\\
    &+
    2\,\COV( \Delta_{m+1}^2(f_i + g_i), 2\Delta_{m+1}(f_i + g_i)\,\Delta_{m+1}(\varepsilon_i) )\notag\\
    &+ 
    2\,\COV( \Delta_{m+1}^2(f_i + g_i), \Delta_{m+1}^2(\varepsilon_i) )\notag\\
    &+ 
    2\,\COV( 2\Delta_{m+1}(f_i + g_i)\Delta_{m+1}(\varepsilon_i), \Delta_{m+1}^2(\varepsilon_i) )\notag\\
    &=
    4\Delta_{m+1}^2(f_i + g_i)\,\gamma_0 + 2\,\gamma_0^2.\label{eq.variance-part1}
\end{align} 

Now, we compute the covariance between the $\Delta_{m+1}^2(y_i)$ terms. 
Since 
$\ME[\Delta_{m+1}^2(y_i)]=\Delta_{m+1}^2(f_i + g_i) + \ME \Delta_{m+1}^2(\varepsilon_i)=\Delta_{m+1}^2(f_i + g_i) + \gamma_0$, 
we can see that
\begin{align}~\label{eq.variance-part2-A}
    \{\Delta_{m+1}^2(y_i) &- \ME[\Delta_{m+1}^2(y_i)] \} \{ \Delta_{m+1}^2(y_j) - \ME[\Delta_{m+1}^2(y_j)] \}
    =
    \Delta_{m+1}^2(\varepsilon_i)\,\Delta_{m+1}^2(\varepsilon_j) \notag\\
    &+
    4\,\Delta_{m+1}(f_i+g_i)\,\Delta_{m+1}(f_j + g_j)\Delta_{m+1}(\varepsilon_{i})\Delta_{m+1}(\varepsilon_{j}) \notag\\ 
    &+
    2[ \Delta_{m+1}(f_i+g_i) \Delta_{m+1}(\varepsilon_i) 
    \Delta_{m+1}^2(\varepsilon_j) + \Delta_{m+1}(f_j+g_j) \Delta_{m+1}(\varepsilon_j) 
    \Delta_{m+1}^2(\varepsilon_i) ] \notag\\
    &-
	2\,\gamma_0 [ \Delta_{m+1}(f_i+g_i) \Delta_{m+1}(\varepsilon_i) 
	+ 
	\Delta_{m+1}(f_j+g_j) \Delta_{m+1}(\varepsilon_j) ] \notag\\
    &-
	\gamma_0 [ \Delta_{m+1}^2(\varepsilon_i) + \Delta_{m+1}^2(\varepsilon_j) ] + \gamma_0^2.	
\end{align}
Due to Gaussianity it follows that for any $i$ and $j$, $\ME[\Delta_{m+1}(\varepsilon_i)] 
= \ME[ \Delta_{m+1}(\varepsilon_i) \Delta_{m+1}^2(\varepsilon_j)  ] = 0.$ In what follows we will need to use the 
following identities concerning central moments of the multivariate normal distribution (see e.g. \cite{Trianta.03}): 
for any integers $r$, $s$, $u$, and $v$
\begin{align}
    \ME[\varepsilon_{r}^{2}\varepsilon_{s}^{2}] 
    &= 
    \gamma_{0}^{2}+2\gamma_{|r-s|}^{2} \notag\\
    \ME[\varepsilon_{r}^{2}\varepsilon_{u}\varepsilon_{v}] 
    &= 
    \gamma_{0}\gamma_{|u-v|}+2\gamma_{|r-u|}\gamma_{|r-v|}\notag\\
    \ME[\varepsilon_{r}\varepsilon_{s}\varepsilon_{u}\varepsilon_{v}] 
    &= 
    \gamma_{|r-s|}\gamma_{|u-v|}+\gamma_{|r-u|}\gamma_{|s-v|}+\gamma_{|r-v|}\gamma_{|s-u|}.~\label{trianta}
\end{align}

Next, observe that
\begin{align*}
    \ME [ \Delta_{m+1}(\varepsilon_i)\,\Delta_{m+1}(\varepsilon_j) ]
    &=
	\gamma_{|j-i|} - \frac{1}{2}\,(\gamma_{|j-i-(m+1)|}+\gamma_{|j-i+(m+1)|}).
\end{align*}
Similar considerations allow us to obtain that
\begin{align*}
    \ME[ \Delta_{m+1}^2(\varepsilon_i)\,\Delta_{m+1}^2(\varepsilon_j) ]
    &=
    \gamma_0^2 
    + 
    2\gamma_{|i-j|}^2
    +
    \frac{1}{2}( \gamma_{|j-i-(m+1)|} + \gamma_{|j-i+m+1|})^2\\
    &-
    2\gamma_{|j-i|}\left( \gamma_{|j-i-(m+1)|} + \gamma_{|j-i+m+1|} \right).
\end{align*}

Due to $m$-dependency $\gamma_{| j-i + m+1 |}=0$ provided that $j-i > 0$. Thus,
taking expectation on both sides of \eqref{eq.variance-part2-A} and utilizing the identities just derived
we get
\begin{align}~\label{eq.variance-part2-B}
    \COV&(\Delta_{m+1}^2(y_i), \Delta_{m+1}^2(y_j))
    = 
    4 \Delta_{m+1}(f_i + g_i) \Delta_{m+1}(f_j + g_j) 
    \left[ \gamma_{|j-i|} - \frac{1}{2} \gamma_{|j-i-(m+1)|} \right] \notag\\
    &+
    2\gamma_{|j-i|}^2 + \frac{1}{2} \gamma_{|j-i-(m+1)|}^2
    -
    2\gamma_{|j-i|}\,\gamma_{|j-i-(m+1)|}.
\end{align}

Substituting \eqref{eq.variance-part1} and \eqref{eq.variance-part2-B} into \eqref{eq.variance}, we get that
  \begin{align}~\label{eq.varianceGeneral}
     &n_{m+1}^2\VAR( \bs{Q}_{1,m+1}(\bs{y}) ) 
     =
     4\gamma_0\sum_{i=1}^{n_{m+1}}\Delta_{m+1}^2(f_i+g_i) + 2 n_{m+1}\gamma_0^2\notag\\
     &+
     8\sum_{i,j}^{(n_{m+1})}\Delta_{m+1}(f_i + g_i) \Delta_{m+1}(f_j + g_j) 
	\left[ \gamma_{|j-i|} - \frac{1}{2} \gamma_{|j-i-(m+1)|} \right]\notag \\
     &+
     4\sum_{i,j}^{(n_{m+1})}\gamma_{|j-i|}^2+\sum_{i,j}^{(n_{m+1})}\gamma_{|j-i-(m+1)|}^2
     -
     4\sum_{i,j}^{(n_{m+1})}\gamma_{|j-i|}\gamma_{|j-i-(m+1)|} \notag \\
     &=
     4\gamma_0n_{m+1}^2\,
     o \left( \left[\frac{(m+1)\,\omega \left( \frac{m+1}{n} \right)}{n_{m+1}} \right]
     \left[ \frac{ \omega( \frac{m+1}{n} ) }{m+1} 
     + 
     \frac{J_K}{2\omega(\frac{m+1}{n})\,n_{m+1}} 
     + 
     \frac{J_K^{||}}{2\,n_{m+1}} \right] \right)(\mbox{see Lemma~\ref{lemma2-proofs}})\notag\\
     &+
     2n_{m+1}\gamma_0^2
     +
     4\gamma_0\,m(m+1)\omega(\frac{m+1}{n})
     \left[ \left(\frac{ n_{m+1}-1 }{m+1}\right)\,
     \omega( \frac{m+1}{n} ) \notag \right.\\
     &+ \left.
     \mc O \left( \frac{J_K}{2 \omega( \frac{m+1}{n} ) } \right)
     +
     \left( \frac{1+\sqrt{2}}{\sqrt{2}} \right)\,J_K^{||} \right] 
     \quad (\mbox{due to Lemma~\ref{lemma3-proofs}})\notag\\
     &-
     2\gamma_0
     \left[ 2m(n_{m+1}-1)\omega^2( \frac{m+1}{n} ) + \frac{m(m+1)}{2}\mc O (J_K) \notag \right.\\
     &+\left.
     \frac{m(m+1)}{\sqrt{2}}\,\omega( \frac{m+1}{n} ) J_K^{||} + \omega( \frac{m+1}{n} )\,\mc O(m H_K^{||})
     \right] (\mbox{due to Lemma~\ref{lemma4-proofs}})\notag\\
     &+
     (2m+1)n_{m+1}\gamma_0^2\quad (\mbox{due to Lemma~\ref{lemma5-proofs}}).
 \end{align}
 After multiplying both sides of \eqref{eq.varianceGeneral} by $n_{m+1}^2$, arraging some terms and doing some algebra, the proof is completed.
\end{pf}\medskip

\begin{pf}[Proof of Theorem~\ref{theo_var_GAMMAh}]
 We will write $\Delta_h$ instead of $\Delta_{1,h}$.
 We begin by using that for $h = 1,\ldots,m$,
 \begin{equation}~\label{eq.var.gammah}
    \VAR( \wh{\gamma}_h ) 
    = 
    \VAR( \bs{Q}_{1,m+1}(\bs{y}; \bs{d}) ) + \VAR( \bs{Q}_{1,h}(\bs{y}) )
    - 2 \COV(\bs{Q}_{1,m+1}(\bs{y}; \bs{d}), \bs{Q}_{1,h}(\bs{y})).
 \end{equation}
 Recall that $\VAR( \bs{Q}_{1,m+1}(\bs{y}; \bs{d}) )$ was established above, cf.~Theorem~\ref{th_var}. 
 Following the ideas and calculations leading 
 to \eqref{eq.variance}-\eqref{eq.varianceGeneral} we find that
 \begin{align*}
    &\VAR( \bs{Q}_{1,h}(\bs{y}) )
    =
    \frac{ 2(\gamma_0 - \gamma_h)^2 }{n_h} + \frac{2(\gamma_0 - \gamma_h)}{n_h^2}\,
    \sum_{i=1}^{n_h}\Delta_h^2(f_i + g_i)  \\
    &+
    \frac{2}{n_h^2}\sum_{i,j}^{(n_h)}\Delta_h(f_i + g_i)\,\Delta_h(f_j + g_j)\, 
    [ 2\gamma_{|j-i|} - \left( \gamma_{|j-i+h|} + \gamma_{|j-i-h|} \right) ] \\ 
    &+
    \frac{1}{n_h^2}\sum_{i,j}^{(n_h)} [ 4 \gamma_{|j-i|}^2 
    + 
    ( \gamma_{|j-i+h|} + \gamma_{|j-i-h|} )^2 - 4 \gamma_{|j-i|}( \gamma_{|j-i+h|} + \gamma_{|j-i-h|} ) ].
 \end{align*}
  
 We conclude this calculation by adapting Lemmas~\ref{lemma2-proofs}, \ref{lemma3-proofs} 
 and \ref{lemma4-proofs} from the Appendix~\ref{sec.appendix} to the current situation. 
 More precisely, following the proof of each of those lemmas, line by line, we can show 
 first that
 \begin{equation*}
  \frac{1}{n_h^2}\sum_{i=1}^{n_h}\Delta_h^2(f_i + g_i) 
  = 
  o\left( \left[ \frac{h \omega( h/n )}{n_h} \right]\,
  \left[ \frac{ \omega(h/n) }{h} 
  + \frac{J_K}{2\,\omega(h/n)\,n_h} 
  + \frac{J_K^{||}}{2\,n_h}  \right]  \right)
 \end{equation*}
 Next, denote $[ \Delta_h(f + g) ]_{i,j} = \Delta_{h}(f_i+g_i)\Delta_{h}(f_j+g_j)$ and verify directly that
  \begin{equation*}
    \frac{4}{n_h^2}\sum_{i,j}^{(n_h)}\,[ \Delta_h( f + g ) ]_{i,j}\,\gamma_{|j-i|} 
    =
o\left( \left[ \frac{\gamma_0mh\omega(h/n)}{n_h} \right]
	\left[ 
	2\frac{\omega(h/n)}{h} 
	+ 
	\frac{J_K}{\omega(h/n)\,n_h}
	+ 
	\sqrt{2}(1+\sqrt{2})\frac{J_K^{||}}{n_h} 
	\right] 
\right)
 \end{equation*}
 and
   \begin{equation*}
   \frac{2}{n_h^2}\sum_{i,j}^{(n_h)}\,
   [ \Delta_h( f + g ) ]_{i,j}\,\gamma_{|j-i-h|} 
   =
	o\left( 
	\left[ 
	\frac{\gamma_0 h \omega(\frac{h}{n})}{n_h}	
	\right]	
	\left[
	2\omega\left(\frac{h}{n}\right)
	+
	\frac{(h+1)\,J_K}{2\omega(\frac{h}{n})n_h}
	+
	\frac{(m+h)\,J_K^{||}}{\sqrt{2}\,n_h}
	+
	\frac{H_{K,h}^{||}}{n_h}
	\right]
	\right)
   \end{equation*}

	An obvious adaptation of Lemma~\ref{lemma5-proofs} allows us
	to see that the third term of $\VAR( \bs{Q}_{1,h}(\bs{y}) )$ above 
	is bounded by $4(4m+h)\gamma_0^2/n_h$.

	Combining the above we get that
 \begin{align}~\label{eq.var.h} 
     &\VAR( \bs{Q}_{1,h}(\bs{y}) )
     =
	\frac{\kappa_1(m,h,\gamma(\cdot))}{n_h}
	+
	o\left( 
	\kappa_2(m,h,\gamma(\cdot))\,\frac{\omega^2(\frac{h}{n})}{n_h}\right)\notag\\ 
	&+
	o\left(
	\kappa_3(n,m,h,\gamma(\cdot))\,\frac{J_K}{n_h}
	+
	\kappa_4(n,m,h,\gamma(\cdot))\,\frac{\omega(\frac{h}{n})\,J_K^{||}}{n_h}
	+
	h\,\frac{\omega(\frac{h}{n})\,H_{K,h}^{||}}{n_h^2}\,\gamma_0
	\right)
 \end{align}
 where
 \begin{align*}
 	\kappa_1(h,m,\gamma(\cdot))
 	&=
 	2(\gamma_0-\gamma_h)^2 + 4(4m + h)\gamma_0^2,\quad
 	\kappa_2(m,h,\gamma(\cdot))
 	=
 	2\left[ (m+h+1)\gamma_0 - \gamma_h \right]\\
 	\kappa_3(n,m,h,\gamma(\cdot))
 	&=
 	\frac{h}{n_h}\,\left[\left( \frac{2m+h+3}{2} \right)\gamma_0 - \gamma_h\right]\\
 	\kappa_4(n,m,h,\gamma(\cdot))
 	&=
 	\frac{h}{n_h}\,
 	\left\{
 	\left[ \frac{(3 + \sqrt{2})m + h + \sqrt{2}}{\sqrt{2}} \right]\gamma_0 
 	- \gamma_h\right\}.
 \end{align*}

 We now move on to the computation of the covariance between 
 $\bs{Q}_{1,m+1}(\bs{y}; \bs{d})$ and $\bs{Q}_{1,h}(\bs{y})$.
	Let us write $A_{m+1} := \ME[\bs{Q}_{1,m+1}(\bs{y}; \bs{d})]$
 and $B_h := \ME[\bs{Q}_{1,h}(\bs{y})]$. Since $\COV(\bs{Q}_{1,m+1}(\bs{y}; \bs{d}), \bs{Q}_{1,h}(\bs{y}))=\ME[ \bs{Q}_{1,m+1}(\bs{y}; \bs{d})\, \bs{Q}_{1,h}(\bs{y}) ] - A_{m+1}B_{h}$,
 and due to Proposition~\ref{prop.1},
 \begin{equation}\label{eq.cov-1-aux1}
	A_{m+1}B_{h}
	=
	\gamma_0( \gamma_0 - \gamma_h )	 + o\left( r_n(m+1, \omega(\cdot), J_K, J_K^{||}) \right)
 \end{equation}
 we must compute the expected value of $\bs{Q}_{1,m+1}(\bs{y}; \bs{d})\, \bs{Q}_{1,h}(\bs{y})$
 in order to complete this part of the proof.
 
 It is readily seen that
\begin{align}\label{eq.cov-1-aux2}
	\ME &[ \bs{Q}_{1,m+1}(\bs{y}; \bs{d})\, \bs{Q}_{1,h}(\bs{y}) ]
	=
	\frac{1}{2p(\bs{d})\,n_{m+1}n_h }\,
	\sum_{i=1}^{n_{m+1}}\,\sum_{j=1}^{n_h}\,
	\ME \left[ \Delta_{m+1}^2(y_i)\,\Delta_{h}^2(y_j)\right]\notag\\
	&=
	\gamma_0(\gamma_0-\gamma_h) +
	o\left( \frac{\gamma_0}{2}
	 \left[h\omega\left(\frac{h}{n}\right)\right]\left[ \frac{\omega(h/n)}{h} 
	 + \frac{J_K}{2\omega(h/n)\,n_h} 
	+ 
	\frac{J_K^{||}}{2n_h} \right]  \right)
	\quad \mbox{(due to Lemma~\eqref{lemma2-proofs})}\notag\\
	&+
	o\left( (\gamma_0-\gamma_h) 
	\left[ (m+1)\omega\left( \frac{m+1}{n} \right)\right] 
	\left[ \frac{\omega\left(\frac{m+1}{n}\right)}{m+1} 
	+ 
	\frac{J_K}{2\omega\left(\frac{m+1}{n}\right)\,n_{m+1}} 
	+ 
	\frac{J_K^{||}}{2n_{m+1}} \right] \right)
	\quad \mbox{(see Lemma~\ref{lemma2-proofs})}\notag\\
	&+
	o\left(  
	\left[ (m+1)\omega\left( \frac{m+1}{n} \right)\right]^2 
	\left[ \frac{\omega\left(\frac{m+1}{n}\right)}{m+1} 
	+ 
	\frac{J_K}{2\omega\left(\frac{m+1}{n}\right)\,n_{m+1}} 
	+ 
	\frac{J_K^{||}}{2n_{m+1}} \right]^2 \right)\notag\\
	&+
	\frac{\gamma_0}{4\sqrt{2}}
	o\left( \left[ \frac{(m+1)\omega\left(\frac{m+1}{n}\right)}{n_{m+1}} \right]
	\left[ \left(\frac{3m+2}{m+1}\right)\omega\left( \frac{m+1}{n} \right) 
	+
	\left( \frac{m+h+1}{2} \right)\frac{J_K}{\omega\left( \frac{m+1}{n} \right)\,n_{m+1}}
	+\right.\right.\notag\\
	&+\left.\left.
	\left( \frac{ (2+\sqrt{2})m+h }{\sqrt{2}} \right)\frac{J_K^{||}}{n_{m+1}} 
	+
	\mc O \left(\frac{H_{K,h}^{||}}{n_h}\right)
	\right] \right) \quad (\mbox{due to Lemma~\ref{lemma3-proofs}})\notag\\
	&+
	8\gamma_0^2/n_{m+1} \quad (\mbox{due to Lemma~\ref{lemma5-proofs}}).
\end{align}
  
	Combining \eqref{eq.cov-1-aux1} and \eqref{eq.cov-1-aux2} we get
\begin{align}\label{eq.cov-1-aux}
	\COV(\bs{Q}_{1,m+1}&(\bs{y}; \bs{d}), \bs{Q}_{1,h}(\bs{y}))
	=
	\frac{8\gamma_0^2}{n_{m+1}}
	+
	\kappa_5(n,m,h,\gamma(\cdot))\,\frac{ \omega^2\left( \frac{m+1}{n} \right) }{n_{m+1}} \notag\\
	&+
	o\left( \kappa_6(n,m,h,\gamma(\cdot))\,\frac{J_K}{n_{m+1}} 
	+
	\kappa_7(n,m,h,\gamma(\cdot))\,\omega\left( \frac{m+1}{n} \right)\,\frac{J_K^{||}}{n_{m+1}}\notag \right.\\
	&+\left.
	\frac{\gamma_0}{4\sqrt{2}}(m+1)\omega\left( \frac{m+1}{n} \right)\,\frac{H_{K,h}^{||}}{n_{m+1}^2} \right),
\end{align}
where 
\begin{align*}
	\kappa_5(n,m,h,\gamma(\cdot)) 
	&= 
	\gamma_0-\gamma_h-1 + \left(\frac{3m+2}{n_{m+1}} \right) \frac{\gamma_0}{4\sqrt{2}} \\
	\kappa_6(n,m,h,\gamma(\cdot))
	&=
	(m+1)\,
	\left\{
	\gamma_0-\gamma_h-1 + \left(\frac{m+h+1}{n_{m+1}}\right) \frac{\gamma_0}{8\sqrt{2}} \right\}\\
	\kappa_7(n,m,h,\gamma(\cdot))
	&=
	(m+1)\left( \frac{\gamma_0-\gamma_h}{2} - \frac{1}{\sqrt{2}} \right)
	+
	\left(\frac{(2 + \sqrt{2})m+h}{n_{m+1}}\right)\,\frac{\gamma_0}{4\sqrt{2}}
\end{align*}

 The proof is complete once we substitute \eqref{eq.VAR.GAMMA0} (see Theorem~\ref{th_var}), 
 \eqref{eq.var.h} and \eqref{eq.cov-1-aux} into \eqref{eq.var.gammah}, and group terms
 appropriately. 
\end{pf}

\section[Appendix B\hfill]{Proofs of Lemmas used in Appendix~\ref{sec.supp}}~\label{sec.appendix}
In addition to the notation introduced in Appendix~\ref{sec.supp}, in this
section for generic vectors $\bs{u}$
and $\bs{v}$ in $\R^d$, we will write $\inner{\bs{u}}{\bs{v}}$ to denote their inner product and $\|\bs{u}\|^2=\inner{\bs{u}}{\bs{u}}$ for the norm of $\bs{u}$.

\begin{lemma}~\label{lemma0}
	Suppose that the conditions of Theorem~\ref{theo.DBE-partial} hold. Then
	for $i\in \{1,n-\ell(m+1)\}$
	\begin{equation}\label{eq-EV-Delta}
		\ME \Delta_{\ell,m+1}^2(y_i; \bs{d})
		=
		\gamma_0\,\left( \sum_{s=0}^\ell\,d_s^2 \right) + o\left( \omega^2\left( \frac{m+1}{n} \right) \right).
	\end{equation}
\end{lemma}
\begin{pf}
	Due to stationarity and $m$-dependence, it can be seen that for any $i$,
	\begin{equation}~\label{eq-EV-Delta-A}
	\ME \Delta_{\ell,m+1}^2(y_i; \bs{d})
	=
	\gamma_0\,\left( \sum_{s=0}^\ell\,d_s^2 \right)
	+
	\left(
	\Delta_{\ell,m+1}(f_i; \bs{d}) + \Delta_{\ell,m+1}(g_i; \bs{d})
	\right)^2.
	\end{equation}
	
	Observe that for any $i$, $\Delta_{\ell,m+1}(f_i; \bs{d})$ can be written as a pseudo
	telescopic sum. Indeed,
	\begin{align*}
	\Delta_{\ell,m+1}(f_i; \bs{d})
	&=
	d_0f_i + d_1f_{i+m+1} + \cdots + d_{\ell-1}f_{i+(\ell-1)(m+1)} + d_lf_{i+\ell(m+1)}\\
	&=
	d_0(f_i-f_{i+m+1})
	+
	(d_0+d_1)f_{i+m+1}+\cdots+d_\ell\,f_{i+\ell(m+1)}\\
	&\vdots\\
	&=
	d_0(f_i-f_{i+m+1})
	+
	(d_0+d_1)(f_{i+m+1}-f_{i+2(m+1)})
	+\cdots\\
	&+
	(d_0+d_1+\cdots+d_{\ell-1})(f_{i+(\ell-1)(m+1)}-f_{i+\ell(m+1)})\\
	&+
	(d_0+\cdots+d_{\ell})(f_{i+\ell(m+1)})\\
	&=
	\sum_{r=0}^{\ell-1}\left( \sum_{s=0}^{r} d_s \right)\,\Delta_{1,m+1}( f_{i+r(m+1)}; (1,-1) ).
	\end{align*}
	The last identity follows because $\sum_{s=0}^l\,d_s=0$. The same arguments
	used above yield
	\begin{equation}~\label{eq-Delta-telescopic-g}
	\Delta_{\ell,m+1}(g_i; \bs{d})
	=
	\sum_{r=0}^{\ell-1}\left( \sum_{s=0}^{r} d_s \right)\,\Delta_{1,m+1}( g_{i+r(m+1)}; (1,-1) ).
	\end{equation}
	In what follows we will suppress the vector $(1,-1)$ from the notation
	hoping to cause no confusion.
	
	Using that for any sequence of real numbers $a_1,\ldots, a_n$,
	\begin{equation*}
	(\sum_{i=1}^n\,a_i)^2
	=
	\sum_{i=1}^n\,a_i^2 + 2 \sum_{j=1}^{n-1}\sum_{k=j+1}^n\,a_j\,a_k,	
	\end{equation*}
	it follows that for any $i$,
	\begin{align}
	\Delta_{\ell,m+1}^2&(f_i;\bs{d})
	=
	\sum_{r=0}^{\ell-1}\left( \sum_{s=0}^r\,d_s \right)^2\,\Delta_{1,m+1}^2(f_{i+r(m+1)})\notag\\
	&+
	2\times \ind_{[2,\infty)}(\ell)\,
	\sum_{r=0}^{\ell-2}\,\sum_{p=r+1}^{\ell-1}\,\left( \sum_{s=0}^r\,d_s \right)\left( \sum_{q=0}^{p} d_q \right)\,
	\Delta_{1,m+1}(f_{i+r(m+1)})\,\Delta_{1,m+1}(f_{i+p(m+1)}).~\label{eq.binSquare.applied}
	\end{align}
	
	Next, observe that for any $i$ and $0\leq r\leq \ell-1$,
	\[
	\Delta_{1,m+1}^2(f_{i+r(m+1)})
	\leq 
	\omega^2\left( \frac{m+1}{n} \right),
	\]
	and similarly, for any $0\leq r \leq \ell-2$ and $r+1\leq p \leq \ell-1$,
	\[
	\left| \Delta_{1,m+1}(f_{i+r(m+1)})\,\Delta_{1,m+1}(f_{i+p(m+1)}) \right|	
	\leq
	\omega^2\left( \frac{m+1}{n} \right).
	\]
	
	Therefore for any $i$,
	\begin{equation}~\label{eq-Delta-f}
	\Delta_{\ell,m+1}^2(f_i)
	\leq 
	\left[ \sum_{r=0}^{\ell-1}\left( \sum_{s=0}^r\,d_s \right)^2 
	+ 2\cdot \ind_{[2,\infty)}(\ell)\,
	\sum_{r=0}^{\ell-2}\,\sum_{p=r+1}^{\ell-1}\,\left( \sum_{s=0}^r\,d_s \right)\left( \sum_{q=0}^{p} d_q \right)
	\right]\,\omega^2\left( \frac{m+1}{n} \right).	
	\end{equation}
	
	Next, let us consider $\Delta_{\ell,m+1}(g_i)$,
	see \eqref{eq-Delta-telescopic-g}. 
	Since for $0\leq r\leq \ell-1$, $t_0 < 1+r(m+1) < 1+(r+1)(m+1) \leq t_1$,
	it follows that $\Delta_{1,m+1}(g_{1+r(m+1)})=g_{1+r(m+1)}-g_{1+(r+1)(m+1)}=0$.
	Hence, $\Delta_{\ell,m+1}(g_1)=0$. Similarly it can be proved that
	$\Delta_{\ell,m+1}(g_{n_{\ell_{m+1}}})=0$.
	
	Since for any $a,b\in \R$, $(a+b)^2 \leq 4 (a^2 \vee b^2)$, we can combine
	\eqref{eq-Delta-f} and the previous paragraph to conclude that for $i\in \{1,n-\ell(m+1)\}$,
	\begin{equation}~\label{eq-EV-Delta-B}
	\left(
	\Delta_{\ell,m+1}(f_i; \bs{d}) + \Delta_{\ell,m+1}(g_i; \bs{d})
	\right)^2
	=
	o\left( \omega^2 \left(\frac{m+1}{n}\right) \right).
	\end{equation}
	The result follows once we combine \eqref{eq-EV-Delta-A} and \eqref{eq-EV-Delta-B}.
\end{pf}\medskip

\begin{lemma}~\label{lemma}
	Suppose that the conditions of Theorem~\ref{theo.DBE-partial} hold. Then,
  \[
	\frac{1}{n_{\ell_{m+1}}}\sum_{i=1}^{n_{\ell_{m+1}-1}}\,
	\bs{f}_{i:(i+1+\ell(m+1))}^\top\,D\,\bs{f}_{i:(i+1+\ell(m+1))}
	=
	o\left( \omega^2\left(\frac{m+1}{n}\right) \right).
  \]  
\end{lemma}
\begin{pf}
	Since for $i \leq n_{\ell_{m+1}}$,
	\begin{align}
		&\bs{f}_{i:(i+1+\ell(m+1))}^\top\,D\,\bs{f}_{i:(i+1+\ell(m+1))}
		=
		\|\tilde{D}\bs{f}_{i:(i+1+\ell(m+1))}\|^2 \notag \\
		&=
		\inner{\bs{w}_\ell^\top}{\bs{f}_{i:(i+\ell(m+1))}}^2 
		+ 
		\inner{\bs{w}_\ell^\top}{\bs{f}_{(i+1):(i+1+\ell(m+1))}}^2 
		=
		\Delta_{\ell,m+1}^2(f_i;\bs{d}) + \Delta_{\ell,m+1}^2(f_{i+1};\bs{d}),\label{eq_lem1_A}
	\end{align} 		
	in what follows we will consider $\Delta_{\ell,m+1}^2(f_i;\bs{d})$
	only as the second term above can be handled similarly. 
	
	From \eqref{eq-Delta-f} we get that
	\begin{equation}\label{eq_lem1_B}
	\frac{1}{n_{\ell_{m+1}}}\,\sum_{i=1}^{n_{\ell_{m+1}}}\,\Delta_{\ell,m+1}^2(f_i; \bs{d})
	=
	o\left( \omega^2 \left( \frac{m+1}{n} \right) \right).
	\end{equation}	
	The result follows once we combine \eqref{eq_lem1_A} and \eqref{eq_lem1_B}.
\end{pf}\medskip

\begin{lemma}~\label{lemmaBIAS1}
	Suppose that the conditions of Theorem~\ref{theo.DBE-partial} hold. Then 
    {
    \begin{align*}
        \sum_{i=1}^{n_{\ell_{m+1}}-1}\,&\bs{g}_{[i:(i+1+\ell(m+1))]}^\top\,D\,\bs{g}_{[i:(i+1+\ell(m+1))]}
        =
		o\left( (m+1)\, J_K\, P_\ell(\bs{d}) \right)
    \end{align*}}
    where for $\ell\geq1$
    \[
        P_\ell(\bs{d})
        =
        \sum_{r=0}^{\ell-1}(\ell-r)\left( \sum_{s=0}^r d_s \right)^2
        +
        2\times\ind_{[2,\infty)}(\ell)\times
        \left(
        \sum_{r=0}^{\ell-2}(\ell-1-r)
        \sum_{s=0}^r\,d_s\,\sum_{p=r+1}^{\ell-1}\sum_{q=0}^{s}\,d_q
        \right).
    \]
\end{lemma}
\begin{pf}
	From \eqref{eq_lem1_A} it can be seen that
	\begin{equation}~\label{eq.GDG}
	\sum_{i=1}^{n_{\ell_m}-1}\,\bs{g}_{i:(i+1+\ell(m+1))}^\top\,D\,\bs{g}_{i:(i+1+\ell(m+1))}
	=
	\sum_{i=1}^{n_{\ell_m}-1}\,\Delta_{\ell,m+1}^2(g_i;\bs{d}) 
	+ 
	\sum_{i=1}^{n_{\ell_m}-1}\,\Delta_{\ell,m+1}^2(g_{i+1};\bs{d}).
	\end{equation}
	In what follows we will consider the first term only
	as the second term on the right-hand side above 
	can be handled similarly. Occasionally, we will suppress $\bs{d}$ from the notation
	hoping to cause no confusion.
 
 It can be seen that
 {
 \begin{align}
 	&\sum_{i=1}^{n_{\ell_m}-1}\,\Delta_{\ell,m+1}^2(g_i)
 	<
 	\sum_{i=1}^{n_{\ell_m}}\,\Delta_{\ell,m+1}^2(g_i)
 	=
 	\sum_{i=0}^{K-1}\,\sum_{j=t_i}^{t_{i+1}}\,\Delta_{\ell,m+1}^2(g_j)\notag\\
 	&=
 	\sum_{i=0}^{K-1}\,
 	\left(
		\sum_{j=t_i}^{t_{i+1}-\ell(m+1)-1}\,\Delta_{\ell,m+1}^2(g_j)
		+
		\sum_{j=t_{i+1}-\ell(m+1)}^{t_{i+1}-1}\,\Delta_{\ell,m+1}^2(g_j)
 	\right).~\label{eq.BOUND.sil}
 \end{align}}
 
	Using \eqref{eq-Delta-telescopic-g} and \eqref{eq.binSquare.applied},
	we can write for any $i \leq n_{\ell_{m+1}}$ and for $\ell\geq 1$
	\begin{equation}~\label{eq.squared-a}
		\Delta_{\ell,m+1}^2(g_i;\bs{d})
		=
		A_{i,\ell}(\bs{d}) + B_{i,\ell}(\bs{d}),
	\end{equation}
	where
 	\begin{align*}
 	A_{i,\ell}(\bs{d})
 	&=
	\sum_{r=0}^{\ell-1}\left(\sum_{s=0}^r\,d_s\right)^2\,\Delta_{1,m+1}^2(g_{i+r(m+1)}),\\
 	B_{i,l}(\bs{d})
 	&=
 	2\cdot \ind_{[2,\infty)}(\ell)\,\sum_{r=0}^{\ell-2}\sum_{p=r+1}^{\ell-1}\,
 	\left(\sum_{s=0}^r\,d_s\right)\,\Delta_{1,m+1}(g_{i+r(m+1)})
 	\left(\sum_{q=0}^p\,d_q\right)\,\Delta_{1,m+1}(g_{i+p(m+1)}).
 	\end{align*}

 	Now, we focus on the first term on the right-hand side of Eq.~\eqref{eq.BOUND.sil}. 
	For the following arguments, suppose that $t_j$ is a generic change-point (fixed). 
	Observe that for any $s\in \{0,\ldots,\ell\}$ and any
 	$i\in\{t_j, t_j+1,\ldots, t_{j+1}-\ell(m+1)-2, t_{j+1}-\ell(m+1)-1\}$,
 	$t_j\leq t_j+s(m+1)\leq i+s(m+1)\leq t_{j+1}+(s-l)(m+1)-1<t_{j+1}$,
 	which implies that $g_{i+s(m+1)}=a_j$. Therefore, $\Delta_{1,m+1}^2(g_{i+r(m+1)})=0$
 	for $i\in[t_j,t_{j+1}-\ell(m+1)-1]$ and $0\leq r\leq \ell-1$. 
 	Similarly, it can be seen that $\Delta_{1,m+1}(g_{i+r(m+1)})\,\Delta_{1,m+1}(g_{i+p(m+1)})=0$	
 	for $i\in\{t_j, t_j+1,\ldots, t_{j+1}-\ell(m+1)-2, t_{j+1}-\ell(m+1)-1\}$, 
 	$0\leq r\leq \ell-2$ and $r+1\leq p\leq \ell-1$.

	Since the above holds true for any generic change-point $t_j$,
	we conclude that
 	\begin{align}~\label{eq.firstBOUND}
 	\sum_{j=0}^{K-1}\, 	
 	\sum_{i=t_j}^{t_{j+1}-\ell(m+1)-1}\Delta_{\ell,m+1}^2(g_i)
	=
	0.
 	\end{align}
 	
 	Next, we move to the second summand on the right-hand side of Eq.~\eqref{eq.BOUND.sil}. 
 	For the following arguments, let us suppose that $t_j$ is a generic 
 	(and fixed) change-point. For $r\in\{0,1,\ldots,\ell-1\}$ define the 
 	sets 
\[
	{\mc C}_{j,r}=\{t_j - (r+1)(m+1),\ldots, t_j - 1 \}.
\] 	
 	Note that when $i$ takes values
 	on the set $\mc{C}_{j,r}$
 	for $r\in\{0,1,\ldots,\ell-1\}$ then both $g_{i+r(m+1)}$ and
 	$g_{i+(r+1)(m+1)}$ are equal to $a_{j-1}$. Due to the definition
 	of $g$, when $i\geq t_j-(r+1)(m+1)$ then $g_{i+(r+1)(m+1)}$ takes
 	the value of the next jump, that is $a_j$. In summary,
 \[
    \Delta_{1,m+1}(g_{i+r(m+1)}) 
    =
    (a_{j-1} - a_j)
    \ind_{{\mc C}_{j,r}}(i),\quad
    r=0,\ldots,\ell-1.
 \]
 Note also that by construction, for $r>s$, $\mc{C}_{j,r} \supset \mc{C}_{j,s}$.

 In light of the above,
\begin{align}~\label{eq.BOUND.A}
	\sum_{i=t_{j+1}-\ell(m+1)}^{t_{j+1}-1}\,A_{i,\ell}(\bs{d})
	&=
	\sum_{i\in \mc{C}_{j+1,\ell-1}}\,A_{i,\ell}(\bs{d})\notag\\
	&=
	\sum_{r=0}^{\ell-1}\,\sum_{i\in \mc{C}_{j+1,r}}\,\left( \sum_{s=0}^r\,d_s \right)^2\,
	(a_j-a_{j+1})^2\,\ind_{\mc{C}_{j+1,r}}(i)\notag\\
	&=
	(a_j-a_{j+1})^2\,(m+1)\,\sum_{r=0}^{\ell-1}\,(\ell-r)\,\left( \sum_{s=0}^r\,d_s \right)^2.
\end{align}
 
	For $\ell\geq 2$, similar considerations yield
\begin{align}~\label{eq.BOUND.B}
	&\sum_{i=t_{j+1}-\ell(m+1)}^{t_{j+1}-1}\,B_{i,l}(\bs{d})
	=
	\sum_{i\in \mc{C}_{j+1,l-1}}\,B_{i,l}(\bs{d}) \notag \\
	&=
	2\,
	\sum_{r=0}^{\ell-2}\,
	\sum_{i\in \mc{C}_{j+1,r}}
	\sum_{s=0}^r\,d_s\,\sum_{p=r+1}^{\ell-1}\sum_{q=0}^p\,d_q\,
	(a_{j}-a_{j+1})^2\,\ind_{\mc{C}_{j+1, r\wedge p }}(i) \notag \\	
	&=
	2\,(a_j-a_{j+1})^2(m+1)\,
	\sum_{r=0}^{\ell-2}(\ell-1-r)\,
	\left\{
	\sum_{s=0}^r\,d_s\,\sum_{p=r+1}^{\ell-1}\sum_{q=0}^p\,d_q
	\right\}
\end{align} 
  
   Combining \eqref{eq.BOUND.A} and \eqref{eq.BOUND.B} 
   we get that
   \[
   \sum_{j=0}^{K-1}\,\sum_{i=t_{j+1}-l(m+1)}^{t_{j+1}-1}\,
   \Delta_{\ell,m+1}^2(g_i;\bs{d}) 
   \leq 
   (m+1)P_\ell(\bs{d})\,\sum_{j=0}^{K-1}(a_j - a_{j+1})^2.
   \]
 The result follows once we combine \eqref{eq.GDG}, \eqref{eq.BOUND.sil}, \eqref{eq.firstBOUND} 
 and the inequality above.
\end{pf}\medskip

\begin{lemma}~\label{lemma4}
	Suppose that the conditions of Theorem~\ref{theo.DBE-partial} hold. Then,
	\begin{equation}
	\sum_{i=1}^{n_{\ell_{m+1}}}\,
	\bs{f}_{i:(i+1+\ell(m+1))}^\top\,D\,\bs{g}_{i:(i+1+\ell(m+1))}
	=
	o \left( \omega\left( \frac{m+1}{n} \right)\, (m+1)\, R_\ell(\bs{d})\, J_K^{||} \right),
	\end{equation}
	where for $\ell\geq 1$,
	\[
	R_\ell(\bs{d})
	=
	\sum_{r=0}^{\ell-1}\,\sum_{s=0}^{r}|d_s|\,
	\sum_{r=0}^{\ell-1}(\ell-r)\,\sum_{s=0}^{r}|d_s|.
	\]
\end{lemma}
\begin{pf}
	It is straightforward that
	\begin{align*}
	\bs{f}_{i:(i+1+\ell(m+1))}^\top\,D\,\bs{g}_{i:(i+1+\ell(m+1))}
	&=
	\inner{\tilde{D}\bs{f}_{i:(i+1+\ell(m+1))}}{\tilde{D}\bs{g}_{i:(i+1+\ell(m+1))}}\\
	&=
	\Lambda(f_i,g_i) + \Lambda(f_{i+1},g_{i+1}),
	\end{align*}
	where
	\[
	\Lambda(f_i,g_i)
	=
	\sum_{r=0}^{\ell-1}\left( \sum_{s=0}^r\,d_s \right)\,\Delta_{1,m+1}(f_{i+r(m+1)})
	\times
	\sum_{r=0}^{\ell-1}\left( \sum_{s=0}^r\,d_s \right)\,\Delta_{1,m+1}(g_{i+r(m+1)}).
	\]
	
	It suffices to consider $\sum_{1\leq i\leq n}\,\Lambda(f_i,g_i)$ as 
	the other term can be handled similarly. 
	
	We begin by noticing that for any $i\leq n_{\ell_{m+1}}$,
	\begin{equation}~\label{eq.Lemma4.A}
	| \Delta_{1,m+1}(f_{i+r(m+1)}) | \leq \omega \left( \frac{m+1}{n} \right).	
	\end{equation}	
		
	Next, following the ideas leading to Eq.~\eqref{eq.firstBOUND}, we
	get
	\begin{equation}~\label{eq.Lemma4.B}
	\sum_{j=0}^{K-1}\,\sum_{i=t_j}^{t_{j+1}-\ell(m+1)-1}\,
	\sum_{r=0}^{\ell-1}\,\sum_{s=0}^{r}|d_s|\, | \Delta_{1,m+1}(g_{i+r(m+1)}) |
	=
	0.
	\end{equation}
	
	Similarly as in \eqref{eq.BOUND.A}, with $t_{j+1}^{\ell,m+1}=t_{j+1}-\ell(m+1)$,  we can see that
	\begin{align}
	\sum_{j=0}^{K-1}\,&\sum_{i=t_{j+1}^{\ell,m+1}}^{t_{j+1}-1}\,\sum_{r=0}^{\ell-1}\,\sum_{s=0}^{r}|d_s|\,
	| \Delta_{1,m+1}(g_{i+r(m+1)}) |
	=
	\sum_{j=0}^{K-1}\,\sum_{i\in\mc{C}_{j+1,\ell-1}}\sum_{r=0}^{\ell-1}\,\sum_{s=0}^{r}|d_s|\,
	|a_j-a_{j+1}|\,\ind_{\mc{C}_{j+1,r}}(i)\notag\\
	&=
	\sum_{j=0}^{K-1}\,|a_j-a_{j+1}|\,(m+1)\,\sum_{r=0}^{\ell-1}(\ell-r)\,\sum_{s=0}^{r}\,|d_s|.~~\label{eq.Lemma4.C}
	\end{align}
	
	Therefore, combining Eqs.\eqref{eq.Lemma4.A}, \eqref{eq.Lemma4.B} and \eqref{eq.Lemma4.C}
	we get
	\begin{align*}
	\left|\sum_{i=1}^{n_{\ell_{m+1}}-1}\Lambda(f_i,g_i)\right|
	&\leq 
	\sum_{i=1}^{n}\left| \Lambda(f_i,g_i) \right|
	=
	\sum_{j=0}^{K-1}\,\sum_{i=t_j}^{t_{j+1}}\left| \Lambda(f_i,g_i) \right|\\
	&\leq
	\omega\left(\frac{m+1}{n}\right)\,\sum_{r=0}^{\ell-1}\,\sum_{s=0}^{r}|d_s|\times
	\sum_{j=0}^{K-1}\,\sum_{i=t_{j+1}^{\ell,m+1}}^{t_{j+1}-1}\,\sum_{r=0}^{\ell-1}\,\sum_{s=0}^{r}|d_s|\,
	| \Delta_{1,m+1}(g_{i+r(m+1)}) |\\
	&\leq
	\omega\left(\frac{m+1}{n}\right)\,\sum_{r=0}^{\ell-1}\,\sum_{s=0}^{r}|d_s|\,
	(m+1)\,
	\sum_{r=0}^{\ell-1}(\ell-r)\,\sum_{s=0}^{r}|d_s|\,J_K^{||},
	\end{align*}
	which completes the proof.
\end{pf}\medskip

\begin{lemma}~\label{lemma2-proofs}
    Suppose that the conditions of model \eqref{eq.partialModel} are satisfied. 
    Suppose also that
    \eqref{eq.DistBetweenJumps} and \eqref{eq.DBE-partial.conditions} hold with $\ell=1$. 
    Then
\begin{align*}
	\frac{1}{n_{m+1}^2}\,&\sum_{i=1}^{n_{m+1}}\,\Delta_{m+1}^2( f_i + g_i, \bs{d} )
	=\\
	&
	o\left( 
	\left[ 
		\frac{(m+1) \omega(\frac{m+1}{n})}{n_{m+1}}
	\right]\,
	\left[ 
	\frac{ \omega(\frac{m+1}{n}) }{m+1} 
	+ 
	P_1(\bs{d})\frac{J_K}{\omega( \frac{m+1}{n} )\,n_{m+1}}
	+ 
	R_1(\bs{d})\frac{J_K^{||}}{n_{m+1}} 
	\right]
	\right).
\end{align*}
\end{lemma}
\begin{pf}
    Since
    $\Delta_{m+1}^2( f_i + g_i, \bs{d} )=\Delta_{m+1}^2(f_i; \bs{d})+\Delta_{m+1}^2(g_i; \bs{d})+2\Delta_{m+1}(f_i; \bs{d})\, \Delta_{m+1}(g_i; \bs{d})$
	we get that
\begin{align*}
	n_{m+1}^{-2}\,\sum_{i=1}^{n_{m+1}}\Delta_{m+1}^2(f_i;\bs{d})
	&=
	n_{m+1}^{-1}\,o \left( \omega^2 \left( \frac{m+1}{n} \right)  \right) \qquad (\mbox{due to Lemma~\ref{lemma}, Eq.~\eqref{eq_lem1_B}})\\
	n_{m+1}^{-2}\,\sum_{i=1}^{n_{m+1}}\,\Delta_{m+1}^2(g_i; \bs{d})
	&\leq
	n_{m+1}^{-2}\,(m+1)P_1(\bs{d})\,J_K \qquad (\mbox{due to Lemma~\ref{lemmaBIAS1}})\\
	n_{m+1}^{-2}\,\sum_{i=1}^{n_{m+1}}\,\Delta_{m+1}(f_i; \bs{d})\,\Delta_{m+1}(g_i; \bs{d})
	&=
	n_{m+1}^{-2}o\left(\omega\left(\frac{m+1}{n}\right)(m+1)\,R_1(\bs{d})J_K^{||}\right),
\end{align*}	
where the last equality follows from Lemma~\ref{lemma4}. The result follows after some algebra.
\end{pf}\medskip

\begin{lemma}~\label{lemma3-1}
    Suppose that the conditions of model \eqref{eq.partialModel} are satisfied. 
    Suppose also that
    \eqref{eq.DistBetweenJumps} and \eqref{eq.DBE-partial.conditions} hold with $\ell=1$. 
    Then
    \begin{align}
		\sum_{i=1}^{n-m-2}\,\sum_{j=i+1}^{i+m}\,| g_{j+m+1} - g_j| 
        &= 
        \mc O( m (m+1) J_K^{||} ) \label{eq.lemma3-1A} \\
        \sum_{i=1}^{n-m-2}\,|g_{i+m+1} - g_i|\, \sum_{j=i+1}^{i+m}\,| g_{j+m+1} - g_j| 
        &=
        \mc O\left( \frac{m(m+1)}{2} J_K \right). \label{eq.lemma3-1B}
    \end{align}
\end{lemma}
\begin{pf}
    We begin by establishing \eqref{eq.lemma3-1A}. Recall that $t_j = \lfloor n \tau_j \rfloor$.
    First note that 
    {\footnotesize
    \begin{equation}~\label{eq.lemma3-1A-aux0}
		\sum_{i=1}^{n-m-2}\,\sum_{j=i+1}^{i+m}\,| g_{j+m+1} - g_j|
        \leq 
        \sum_{j=1}^K\,A_j,\quad\mbox{ where }\quad
        A_j
        =
        \sum_{i=t_{j-1}}^{t_j-1}\,\sum_{l=i+1}^{i+m}\,| g_{l+m+1} - g_l|.
    \end{equation}}
    For any $t_j$ with $j\geq 1$, let $\kappa_j = t_j-2(m+1)$.
    We split the computation of $A_j$ in three ways, first we consider
    $i \leq \kappa_j$, then $\kappa_j+1 \leq i \leq \kappa_j+m$ and finally
    $\kappa_j+m+1\leq i \leq t_j-1$.\medskip
    
    Assume that $i\leq \kappa_j$. Note that for
    $l=i+1 \leq t_j-(2m+1)$, $l+m+1 \leq t_j-m$ which implies that $g_{l+m+1}-g_l=0$.
    Note also that, 
    $l=i+m\leq t_j-m-2$ implies that $l+m+1\leq t_j-1$, i.e., $g_{l+m+1}-g_l=0$. 
    Since these arguments hold for any $l=i+1,\ldots,i+m$ when $i\leq \kappa_j$ we have shown that,
    \[
        \sum_{i\leq \kappa_j}\sum_{l=i+1}^{i+m}| g_{l+m+1} - g_l| = 0.
    \]

    Assume now that $\kappa_j+1 \leq i \leq \kappa_j+m$. The arguments presented above
    allow us to get that for $s = 0,1,\ldots, m-1$,
    \[
        \sum_{i=\kappa_j+1+s}^{\kappa_j+1+s}\,\sum_{l=i+1}^{i+m}| g_{l+m+1} - g_l| = (s+1)|a_{j-1} - a_j|.
    \]
    
    Next, assume that $i\in I_{t_j} = \{ \kappa_j+(m+1), \ldots, t_j-1 \}$. 
    With the arguments utilized so far (basically inspection case by case),
    it is not difficult to see  that
    \begin{equation}~\label{eq.lemma3-1A-aux3}
        \sum_{i\in I_{t_j}}\,\sum_{l=i+1}^{i+m}\,|g_{l+m+1} - g_l| = \frac{m(m+1)}{2} \,  |a_{j-1} - a_j|.
    \end{equation}
    
    Eq.~\eqref{eq.lemma3-1A} is established by noticing that
    the righ-hand side of \eqref{eq.lemma3-1A-aux0} is equal to
    \[
        \sum_{j=1}^K\sum_{i=\kappa_j+1}^{t_j-1}\sum_{l=i+1}^{i+m}| g_{l+m+1} - g_l|
        =
        [ 1 + 2 + \cdots + m + \frac{m(m+1)}{2} ]\sum_{j=1}^K\,| a_{j-1} - a_j|.
    \]       
    
    We can utilize the arguments presented above to show Eq.~\eqref{eq.lemma3-1B}. Indeed,
    note first that since $t_j = \lfloor n\tau_j \rfloor$, it is not difficult to see that 
    for $j=1,\ldots, K$,
    \begin{equation}~\label{eq.difference.g}
        g_{i+m+1}-g_i
        =
        \begin{cases}
            0 & \mbox{ for }\tau_{j-1}\leq i\leq \tau_j-(m+2)\\
            a_j - a_{j-1} & \mbox{ for }\tau_{j}-(m+1)\leq i\leq \tau_j-1\\
        \end{cases}.
    \end{equation}
       
    Next, the left-hand side of \eqref{eq.lemma3-1B} is bounded by
    \begin{align*}
    	\sum_{j=1}^K\,\sum_{i=t_{j-1}}^{t_j-1}\,&| g_{i+m+1} - g_i|\,\sum_{l=i+1}^{i+m}| g_{l+m+1} - g_l|\\
    	&\leq 
    	\sum_{j=1}^K\,\sum_{i\in I_{t_j}}| a_{j-1} - a_j |\,\sum_{l=i+1}^{i+m}| g_{l+m+1} - g_l| \qquad (\mbox{by Eq.~}\eqref{eq.difference.g})\\
    	&=
    	\sum_{j=1}^K\,|a_{j-1}-a_j|\,\sum_{i\in I_{t_j}}\,\sum_{l=i+1}^{i+m}| g_{l+m+1} - g_l|\\
    	&=
    	\frac{m(m+1)}{2} \sum_{j=1}^K\,(a_{j-1} - a_j)^2, \qquad (\mbox{by Eq.~}\eqref{eq.lemma3-1A-aux3}).
    \end{align*}                
    This completes the proof.
\end{pf}\medskip

\begin{lemma}~\label{lemma3-proofs}
    Suppose that the conditions of Lemma~\ref{lemma3-1} are satisfied. 
    Then,
\[
	\sum_{i=1}^{n_{m+1}-1}\,\sum_{j=i+1}^{n_{m+1}}\,\Delta_{m+1}(f_i + g_i)\,\Delta_{m+1}(f_j + g_j) \, \gamma_{j-i}
	=
	d_{1}^{2}\gamma_{0}m(m+1)\,\Xi_n( \omega(\cdot), m, d_1, J_K, J_K^{||} ),
\]        
    where 
\[
	\Xi_n( \omega(\cdot), m, d_1, J_K, J_K^{||} )
	=
	\left(\frac{n-m-2}{m+1}\right)\omega^2\left( \frac{m+1}{n} \right)
	+
	\mc{O}\left( \frac{J_K}{2} \right)
	+
	(1+|d_1|)\,\omega\left( \frac{m+1}{n} \right)J_K^{||}.
\]
\end{lemma}
\begin{pf}
    Observe that due to $m$-dependency,
    {\footnotesize
    \[
        \sum_{i=1}^{n_{m+1}-1}\,\sum_{j=i+1}^{n_{m+1}}\,
        \Delta_{m+1}(f_i + g_i)\,\Delta_{m+1}(f_j + g_j) \gamma_{j-i}
        =
        \sum_{i=1}^{n-m-2}\,
        \sum_{j=i+1}^{i+m}\,
        \Delta_{m+1}(f_i + g_i)\,\Delta_{m+1}(f_j + g_j) \gamma_{j-i}.
    \]}
    
    Then, by definition 
    {
    \begin{align*}
        \Delta_{m+1}(f_i + g_i)\,\Delta_{m+1}(f_j + g_j)
        &= 
        \Delta_{m+1}(f_i)\,\Delta_{m+1}(f_j) 
        + 
        \Delta_{m+1}(f_i)\,\Delta_{m+1}(g_j)\\ 
        &+ 
        \Delta_{m+1}(g_i)\,\Delta_{m+1}(f_j) 
        + 
        \Delta_{m+1}(g_i)\,\Delta_{m+1}(g_j).
    \end{align*}}
    Because $d_0+d_1=0$, $\Delta_{m+1}(f_i)\,\Delta_{m+1}(f_j) = d_1^2 (f_{i+m+1} - f_i) (f_{j+m+1} - f_j)$. 
	The uniform continuity of $f$ ensures that
    \[
        | \Delta_{m+1}(f_i)\,\Delta_{m+1}(f_j) | 
        \leq 
        d_1^2\omega^{2}\left(\frac{m+1}{n}\right).
    \]
    Similarly, we find that 
    \begin{align*}
        | \Delta_{m+1}(f_i)\,\Delta_{m+1}(g_j) | 
        &\leq 
        d_1^2\omega\left(\frac{m+1}{n}\right)\left\vert g_{j+m+1} - g_j \right\vert,\\
		| \Delta_m(g_i)\,\Delta_m(g_j) | 
        &= 
        d_1^2| g_{i+m+1} - g_i |\,| g_{j+m+1} - g_j |.
    \end{align*}
        
    Consequently,
    \begin{align}
        \sum_{i=1}^{n-m-2}\,\sum_{j=i+1}^{i+m}\,        
        &\,
        \left\vert \Delta_{m+1}(f_i + g_i)\,\Delta_{m+1}(f_j + g_j) \frac{\gamma_{j-i}}{\gamma_0} \right\vert
        \leq
        d_1^2\,m(n-m-2)\,\omega^{2}\left(\frac{m+1}{n}\right)\notag\\
        &+
        d_1^2\omega\left(\frac{m+1}{n}\right)
        \left\{ \sum_{i,j}| g_{j+m+1} - g_j | 
        + 
        m \sum_{i=1}^{n-m-2}\,| g_{i+m+1} - g_i | \right\} \notag \\ 
        &+
        d_1^2\,\sum_{i=1}^{n-m-2}\,|g_{i+m+1} - g_i|\, 
        \sum_{j=i+1}^{i+m}\,| g_{j+m+1} - g_j|.\label{eq.lemma3-aux} 
    \end{align}

	Next, for a generic and fixed change-point $t_i$, set $\kappa_i=t_i-2(m+1)$ and
	use Lemma~\ref{lemma3-1} to get that
\begin{equation}\label{eq.lemma3-A}
	\sum_{i,j}| g_{j+m+1} - g_j |
	\leq 
	\sum_{i=1}^K\,\sum_{j=\kappa_i+1}^{t_i-1}\,\sum_{l=j+1}^{j+m}| g_{l+m+1} - g_l|
	=
	m(m+1)\,J_K^{||} \qquad (\mbox{due to Eq.~\eqref{eq.lemma3-1A}}).
\end{equation}

	Similarly, we get that
\begin{align}
	\sum_{i=1}^{n-m-2}| g_{i+m+1} - g_i|
	&\leq 
	\sum_{j=1}^K\,\sum_{i=t_{j-1}}^{t_j}| g_{i+m+1} - g_i | \notag\\
	&\leq 
	|d_1|\,\sum_{j=1}^K\,\sum_{i=t_j-(m+1)}^{t_j-1}| a_{j-1} - a_j |\qquad (\mbox{due to Eq.~}\eqref{eq.difference.g}) \notag\\
	&=
	(m+1)|d_1|\,\sum_{j=1}^{K}| a_{j-1} - a_j | = (m+1)|d_1|J_K^{||}.\label{eq.lemma3-B}
\end{align}

The last summand of \eqref{eq.lemma3-aux} is of order $\mc{O}(m(m+1)\,J_K/2)$ according to
Lemma ~\ref{lemma3-1}, Eq.~\eqref{eq.lemma3-1B}. The result follows once we substitute the latter bound
along with \eqref{eq.lemma3-A} and \eqref{eq.lemma3-B} into \eqref{eq.lemma3-aux}.
\end{pf}\medskip

\begin{lemma}~\label{lemma4-1}
    Suppose that the conditions of model \eqref{eq.partialModel} are satisfied. 
    Suppose also that
    \eqref{eq.DistBetweenJumps} and \eqref{eq.DBE-partial.conditions} hold with $\ell=1$. 
    Then
    \begin{align}
        \sum_{i=1}^{n_{m+1}-1}\,\sum_{j=i+1}^{i+2m+1}\,| g_{j+m+1} - g_j| 
        &= 
        \mc O\left( m \, H_K^{||} \right)\label{eq.lemma4-1A} \\
        \sum_{i=1}^{n_{m+1}-1}\,|g_{i+m+1} - g_i|\, \sum_{j=i+1}^{i+2m+1}\,| g_{j+m+1} - g_j| 
        &=
        \mc O\left( \frac{m (m+1)}{2}\,J_K \right), \label{eq.lemma4-1B}
    \end{align}
    where $H_K^{||} = \sum_{j=1}^{K} ( t_j - \frac{m+1}{2} ) | a_{j-1} - a_j|$.
\end{lemma}
\begin{pf}
    We begin by establishing \eqref{eq.lemma4-1A}. Recall that $t_j = \lfloor n\tau_j \rfloor$.
    As in Lemma~\ref{lemma3-1} here we also split the sum
    in three ways: let $\kappa_j = t_j-2(m+1)$, first we consider $i\leq \kappa_j$, then
    $\kappa_j+1 \leq i \leq \kappa_j+m$ and finally $\kappa_j+m+1 \leq i \leq t_j-1$.
    
    Assume that $i\leq \kappa_j$. It is not difficult to see that when $s=m+1,\ldots,2m+1$, 
    and $l=i+s$, $g_{l+m+1}-g_l=a_{j-1}-a_j$.
    Also, when $s\leq m$ and $l=i+s$, $g_{l+m+1}-g_l=0$. Hence,
    {\footnotesize
    \begin{align*}
        \sum_{i\leq \kappa_j}\sum_{l=i+1}^{i+2m+1}|g_{l+m+1}-g_l| 
        &= 
        \sum_{i\leq \kappa_j}\sum_{l=i+m+1}^{i+2m+1}|g_{l+m+1}-g_l|\\
        &=
        m\sum_{i\leq \kappa_j}\,|a_{j-1} - a_j|
        =
        \kappa_j \,m\,|a_{j-1} - a_j|.
    \end{align*}}    
    Next, for any $\kappa_j+1 \leq i \leq \kappa_j+m$, it is straightforward to see that 
    \[
        \sum_{l=i+1}^{i+2m+1}| g_{l+m+1} - g_l| = (m+1)|a_{j-1} - a_j|.
    \]
    
    Assume now that $\kappa_j+m+1 \leq i \leq t_j-1$. We begin by studying the particular
    case $i = \kappa_j+m+1 = t_j-(m+1)$. Then we get for $t=0,\ldots,m-1$ and $l=i+1+t$ 
    that $g_{l+m+1}-g_l = (a_{j-1} - a_j)$.
    Also, when $t\geq m$ and $l=i+1+t$, $g_{l+m+1}-g_l=0$. Hence,
    \[
        \sum_{l=i+1}^{i+2m+1}|g_{l+m+1}-g_l|
        =
        \sum_{l=i+1}^{i+m}|a_{j-1} - a_j|
        =
        m\,|a_{j-1} - a_j|
    \]
    Similar arguments allow us to see that for $i = \kappa_j + m + 1 + u$ where $u = 1, \ldots, m-1$,
    \begin{equation}~\label{eq.lemma4-1.aux}
    \sum_{l=i+1}^{i+2m+1}|g_{l+m+1}-g_l| = (m - u)|a_{j-1} - a_j|.
    \end{equation}
    
    Combining the arguments above, we have shown that the left-hand side of \eqref{eq.lemma4-1A}
    is bounded by
    \begin{align*}
        \sum_{j=1}^K&
        \left[ \sum_{i\leq \kappa_j} + \sum_{i=\kappa_j+1}^{\kappa_j+m} + \sum_{i=\kappa_j+m+1}^{\tau_j-1} \right]
        \sum_{l=i+1}^{i+2m+1}|g_{l+m+1} - g_l|\\
        &=
        m\,\sum_{j=1}^K\,\left[ t_j - 2(m+1) + (m+1) + \frac{m+1}{2} \right]|a_{j-1} - a_j|.
    \end{align*}
    
    In order to establish Eq.~\eqref{eq.lemma4-1B} we follow the proof of Eq.~\eqref{eq.lemma3-1B}
    and apply Eqs.~\eqref{eq.difference.g} and \eqref{eq.lemma4-1.aux}.
    This completes the proof.
\end{pf}\medskip

\begin{lemma}~\label{lemma4-proofs}
    Suppose that the conditions of Lemma~\ref{lemma4-1} are satisfied. 
    Then,
    \begin{align*}
        &\sum_{i=1}^{n_{m+1}-1}\sum_{j=i+1}^{n_{m+1}}\,
        \Delta_{m+1}(f_i + g_i)\,\Delta_{m+1}(f_j + g_j) \gamma_{|j-i-(m+1)|}\\
        &=
        d_{1}^{2}\gamma_{0}
        \left[
        2m(n-m-2) \omega^{2}\left(\frac{m+1}{n}\right)
        +
        \frac{m(m+1)}{2} \mc O\left(J_{K}\right)\right.\\
        &\left.+
        |d_1|m(m+1)\omega\left(\frac{m+1}{n}\right)J_{K}^{||}
        +
        \omega\left(\frac{m+1}{n}\right) \mc O\left(mH_{K}^{||}\right)
        \right]        
 \end{align*}       
   
\end{lemma}
\begin{pf}
	In this proof, we will denote $a_k=\Delta_{m+1}(f_k + g_k)$.
    For given $i$, due to $m$-dependency we get that $\gamma_{|j-i-(m+1)|}\neq 0$ when $j \leq i + 2m+1$.
    Consequently,
\begin{align*}
	&\left| \sum_{i=1}^{n_{m+1}-1}\,\sum_{j=i+1}^{n_{m+1}}\,a_i\,a_j\,\gamma_{|j-i-(m+1)|} \right|
	\leq
	\gamma_0\,\sum_{i=1}^{n_{m+1}-1}\,\sum_{j=i+1}^{i+2m+1}\,\left|a_i\,a_j \right| \notag \\
	&\leq
	\gamma_0\left\{
	d_1^2\,(2m)\,(n-m-2)\,\omega^2\left( \frac{m+1}{n} \right) 
	+	
	d_1^2\,\omega\left( \frac{m+1}{n} \right) 	
	\left[ 
	\sum_{i,j}| g_{j+m+1} -g_j | \notag \right.\right.\\
	&\left.\left.+ 	
	(2m+1)\sum_{i=1}^{n_{m+1}-1}| g_{i+m+1} -g_i | \right]
	+
	d_1^2\sum_{i=1}^{n_{m+1}-1}| g_{i+m+1} -g_i|\,\sum_{j=i+1}^{i+2m+1}|g_{j+m+1}-g_j|	 
	\right\} (\mbox{see Eq.~\ref{eq.lemma3-aux}}) \notag \\
	&=
	d_1^2\gamma_0 
	\left\{ 
	2m(n-m-2)\omega^2\left( \frac{m+1}{n} \right)\notag
	\right.\\
	&\left.+
	\omega\left( \frac{m+1}{n} \right)\left[ \mc O (m H_K^{||}) + |d_1|m(m+1) J_K^{||} \right]
	\quad (\mbox{due to Eq.\eqref{eq.lemma4-1A} and Eq.~\eqref{eq.lemma3-B}, respec.})\notag \right.\\
	&\left.+
	\mc O\left( \frac{m(m+1)}{2} J_K \right)
	\right\} (\mbox{due to Lemma~\ref{lemma4-1}, Eq.\eqref{eq.lemma4-1B}}).
\end{align*}                    
    This completes the proof.
\end{pf}\medskip

\begin{lemma}~\label{lemma5-proofs}
    Let $\gamma_h$ denote the autocovariance function of a stationary, $m$-dependent process. Then
    \begin{enumerate}
        \item $\sum_{i,j} \gamma_{|j-i|}^2 \leq \gamma_0^2(n_m-1)m$,
        
        \item $\sum_{i,j} \gamma_{|j-i-(m+1)|}^2 \leq \gamma_0^2 (n_m-1) ( 2m+1 )$,
        
        \item $\sum_{i,j} |\gamma_{|j-i|} \gamma_{|j-i-(m+1)|}| \leq \gamma_0^2(n_m-1)m$.
    \end{enumerate}
    Here $\sum_{i,j} = \sum_{i=1}^{n_m-1}\,\sum_{j=i+1}^{n_m}$.
\end{lemma}
\begin{pf}
    \begin{enumerate}
        \item $\sum_{i,j} \gamma_{|j-i|}^2 
        = 
        \gamma_0^2\sum_{i,j}(\frac{\gamma_{j-i}}{\gamma_0})^2 \leq \gamma_0^2 \sum_{i=1}^{n_m-1}\sum_{j=i+1}^{i+m}\,1 
        = 
        \gamma_0^2 (n_m-1) m$.
    The inequality follows from the $m$-dependency.
    
        \item First note that
        \[
            |j-i-(m+1)| 
            = 
            \begin{cases}
                j-i-(m+1) & \mbox{ for }j\geq i+(m+1) \\
                i+m+1-j & \mbox{ for }j<i+m+1
            \end{cases}.
        \]
        Then, recall that $\gamma_{|j-i-(m+1)|}\neq 0$ when $|j-i-(m+1)|\leq m$. Intersecting these two subsets we get that
        $\gamma_{|j-i-(m+1)|}\neq 0$ for $i+1\leq j \leq i+2m+1$. The rest of the proof is similar to that of {\bf 1.}
        
        \item Follows from {\bf 1} and {\bf 2}.
    \end{enumerate}
\end{pf}

\bibliographystyle{apalike}
\bibliography{MultiscaleBibUpdate}  

%
\end{document}